\documentclass[12pt]{amsart}
\usepackage{fullpage}

\usepackage[colorlinks=true,urlcolor=blue,
bookmarks=true,bookmarksopen=true,citecolor=blue
]{hyperref}

\usepackage{amsmath, amsfonts, amssymb}      
\usepackage{color} 
\usepackage[all]{xy}
\usepackage{arydshln}
\usepackage{blkarray}
\usepackage{graphicx}

\theoremstyle{plain}
\newtheorem{mainthm}{Theorem}
\newtheorem{thm}{Theorem}[subsection]

\newtheorem{cor}[thm]{Corollary}

\newtheorem{lem}[thm]{Lemma}
\newtheorem{prop}[thm]{Proposition}
\newtheorem{cnj}{Conjecture}

\theoremstyle{definition}
\newtheorem{dfn}[thm]{Definition}

\theoremstyle{remark}
\newtheorem{rem}[thm]{Remark}
\newtheorem*{remnonum}{Remark}
\newtheorem{rems}[thm]{Remarks}
\newtheorem*{remsnonum}{Remarks}

\newtheorem*{exsnonum}{Examples}
\newtheorem{qs}[thm]{Question}

\newtheorem*{qssnonum}{Questions}

\theoremstyle{plain}

\newcommand{\R}{\mathbb{R}}
\newcommand{\Z}{\mathbb{Z}}

\newcommand{\C}{\mathbb{C}}
\newcommand{\F}{\mathbb{F}}
\newcommand{\K}{\overline{K}_1} 

\newcommand{\Crit}{\textnormal{Crit\/}}
\newcommand{\tor}{\text{Tor}}

\newcommand{\Ccal}{\mathcal{C}}

\newcommand{\fcaddress}{francois.charette8@gmail.com}

\begin{document}

\title[Quantum Reidemeister torsion]{Quantum Reidemeister torsion, open Gromov-Witten invariants and a spectral sequence of Oh}
\date{\today}

\author{Fran\c{c}ois Charette}
\thanks{The author was supported by ETH Z\"{u}rich, and an MPIM Bonn fellowship}

\address{Fran\c{c}ois Charette, University of Ottawa, Canada} \email{\fcaddress}

%

\begin{abstract}
  We adapt classical Reidemeister torsion to monotone Lagrangian submanifolds using the pearl complex of Biran and Cornea.  The definition involves generic choices of data and we identify a class of Lagrangians for which this torsion is invariant and can be computed in terms of genus zero open Gromov-Witten invariants.  This class is defined by a vanishing property of a spectral sequence of Oh in Lagrangian Floer theory.
\end{abstract}

\maketitle

%
%

\section{Introduction and results}
Lagrangian quantum homology $QH(L)$ is an invariant associated to a closed monotone Lagrangian submanifold in a tame symplectic manifold $(M, \omega)$, see \cite{Bi-Co:rigidity}.  This invariant can vanish, for example if the Lagrangian can be displaced from itself by a Hamiltonian isotopy, as is the case for any Lagrangians in $\C^n$, in which case $L$ is called narrow.  A variant of this construction associates a quantum homology to every field representation $\varphi\colon H_2(M, L) \to \mathbb{F}^\times$, denoted by $QH^\varphi(L)$, and it often happens that the latter vanishes for most $\varphi$ (e.g.\ monotone toric fibers), even when $L$ is not displaceable.

All known examples of Lagrangians behave in one of the following two ways:  either $QH^\varphi(L)\cong H(L;\mathbb{F})$, where $H(L)$ is the singular homology of $L$ and the isomorphism is not assumed to be canonical.  In this case, $L$ is said to be $\varphi$-wide.  The other extreme possibility is $QH^\varphi(L) = 0$, in which case $L$ is called $\varphi$-narrow and $\varphi$ is a narrow representation.  This latter class of Lagrangians is the main character of this paper.  Relevant definitions will be given in the next sections.

Classically, Reidemeister torsion is an invariant extracted from the $\Z[\pi_1(X)]$-equivariant cellular complex of the universal cover $\tilde{X}$ of a CW-complex $X$.  It is a secondary invariant, meaning that it is defined precisely for representations $\pi_1(X) \to \mathbb{F}^\times$ for which the equivariant cellular homology vanishes.  In particular, the Euler characteristic of $X$ vanishes, $\chi(X; \mathbb{F})=0$. It has been used to classify 3-dimensional lens spaces up to homeomorphism and plays an important role in the classification of manifolds, via the s-cobordism theorem.  For an account of that theory and more, see for example Cohen \cite{Coh:simple}, Milnor \cite{Mil:torsion} or Turaev \cite{Tur:torsionbook}.

In this paper, we adapt Reidemeister torsion to the pearl complex of $\varphi$-narrow Lagrangians, which we call \textit{quantum Reidemeister torsion}.  This is rather direct over a field, although the homological algebra involved is that of periodic chain complexes (i.e.\ with a $\Z / 2 \Z$ grading), rather than bounded chain complexes, as explained in \S \ref{sec:torsion} and \ref{sec:pearlcomplex}.

Somewhat surprisingly, when trying to adapt the construction over \textit{rings}, in order to recover the more general notion of Whitehead torsion, one runs into algebraic difficulties due to the cyclic grading and it is not always possible to define a notion of torsion for such complexes, at least it is not clear to me how to do so.  See \S \ref{sec:whiteheadtorsion} and \ref{sec:pearlboundary} for more on this.

Other notions of torsion have been used in symplectic topology, although they were concerned with the Floer complex rather than the pearl complex, in the works of Abouzaid-Kragh \cite{Abou-Kragh:simplehomtypenearby}, Fukaya \cite{Fu:scobordism}, Lee \cite{Lee:lagtorsion, Lee:torsion1, Lee:torsion2}, Suarez \cite{Su:exactcob} and Sullivan \cite{SullM:kinvariants}.

\subsection{Main results}
Let us first emphasize the following assumption, which will hold throughout the text: \textbf{whenever we write a Lagrangian $L$, we always mean $L$ with a fixed orientation and a fixed spin structure.}

Consider $(M, \omega)$ a symplectic manifold, convex at infinity whenever it is not compact, with a closed, orientable and spin Lagrangian submanifold $L$.  We assume furthermore that $L$ is monotone, which implies that its minimal Maslov index is even and satisfies $N_L \geq 2$, as $L$ is orientable.  More precise definitions will be given from \S \ref{sec:torsion} onwards.

Our definition of torsion involves various choices of generic data $\mathcal{D} = (f, \rho, J)$, consisting of a Riemannian metric $\rho$ on $L$, a Morse function $f\colon L \to \R$ such that $(f, \rho)$ is Morse-Smale, and an almost-complex structure $J$ compatible with $\omega$.  The resulting torsion is denoted by $\tau((L, \varphi), \mathcal{D}) \in \K(\F)$, see \S \ref{sec:pearlcomplex}.  Here, $\varphi\colon H_2(M, L; \Z) \to \mathbb{F}^\times = GL(1, \mathbb{F})$ is a one dimensional representation of $H_2(M, L)$ and we assume that the characteristic of $\mathbb{F}$ does not divide the order of the torsion subgroup of $\oplus_i H_i(L;\Z)$, denoted by $\tor(L)$.   The multiplicative group $\K(\F) = \F^\times / \pm 1$ is called the reduced Whitehead group of $\F$.  Because torsion is defined at the chain-level, it is not a priori clear whether it depends on the choice of generic triple $\mathcal{D}$.  In the present paper, we describe a class of Lagrangians for which we can prove invariance and compute the torsion in terms of genus zero open Gromov-Witten invariants of $L$.  We postpone the general study of invariance to a future paper.    Recall from the work of Oh \cite{Oh:spectral} that to any monotone Lagrangian $L$, we can associate a spectral sequence $(E^*, d^*)$ using the Floer complex and an energy filtration.  Moreover, this can also be twisted by $\varphi$, see \S\ref{sec:d1complex}.  We say that $L$ is \textit{$E^{1, \varphi}$-narrow} if the first page of the twisted spectral sequence has no homology for a representation $\varphi\colon H_2(M, L) \to \mathbb{F}^\times$, which implies that $\varphi$ is a narrow representation; see Definition \ref{def:e1narrow}.  Our main result is as follows.  
\begin{mainthm}\label{thm:first}
Let $L$ be a closed, monotone, orientable, and spin Lagrangian submanifold.  Let $\F$ be a field such that $\text{char } \F \nmid |\tor(L)|$, $\varphi\colon H_2(M, L;\Z) \to \F^\times$ a representation and $\mathcal{D}$ a generic pearl data triple.  Assume that $(L, \varphi)$ is an $E^{1, \varphi}$-narrow pair.  Then the quantum Reidemeister torsion of $(L, \varphi)$ is independent of the choice of generic pearl data $\mathcal{D} = (f, \rho, J)$ and satisfies
$$\tau((L, \varphi), \mathcal{D}) = \dfrac{|\tor_{\text{ev}}(L)|}{|\tor_{\text{odd}}(L)|}\tau(E^{1, \varphi}, h_* \otimes 1) = \tau(E^{0, \varphi}, \text{Crit}_* f, h_*\otimes 1)\tau(E^{1, \varphi}, h_*\otimes 1)$$
Moreover, $\tau(E^{1, \varphi}, h_*\otimes 1)$, as a function of $\varphi$, can be expressed as a rational function of genus zero open Gromov-Witten invariants of $L$, weighted by $\varphi$.
\end{mainthm}
Some explanations are in order.  Here, $\tor_{\text{ev}}(L)$ (resp.\ $\tor_{\text{odd}}(L)$) is the torsion subgroup of $\oplus_{k} H_{2k}(L;\Z)$ (resp.\ $\oplus_{k} H_{2k+1}(L;\Z)$), and their order are  non-zero elements of $\F$, by assumption on the characteristic.   $\tau(E^{0, \varphi}, \text{Crit}_* f, h_*\otimes 1)$ is Milnor's torsion of page $0$ from Oh's spectral sequence, see \S \ref{sec:milnorsdef} and \S \ref{sec:e1narrow}.  Open Gromov-Witten invariants are defined in \S \ref{sec:openGW}.  Finally, $\tau(E^{1, \varphi}, h_*\otimes 1)$ is the torsion of the first page, see \S \ref{sec:torsion} and \ref{sec:e1narrow}.  The theorem will be proved and stated again in Theorem \ref{thm:main}, once all the relevant definitions are introduced.

Given $g \in \text{Symp}(M, \omega)$ and a spin Lagrangian $L$, we get a new spin Lagrangian $g(L)$, that we endow with the orientation and spin structure induced from the ones on $L$ by $g$.  There is also an action on representations $\varphi\colon H_2(M, L;\Z) \to \F^\times$, defined by $\varphi \circ g^{-1}_* \colon H_2(M, g(L)) \to \F^\times$.  In \S \ref{sec:e1narrow}, we will prove:
\begin{cor}\label{cor:thmfirst}
Fix $g \in \text{Symp}(M, \omega)$.  Under the assumptions of Theorem \ref{thm:first}, $g(L)$ is $E^{1, \varphi \circ g^{-1}_*}$-narrow and quantum Reidemeister torsion is invariant under the action of $g$, that is, $\tau(L, \varphi) = \tau(g(L), \varphi \circ g^{-1}_*)$.  In particular, torsion is an invariant of the Hamiltonian isotopy class of $L$.
\end{cor}

Let us now introduce a special class of $E^{1, \varphi}$-narrow Lagrangians for which we can do very explicit calculations.  More details are given in \S \ref{sec:assumptionstar}.  Assume that $L=S^{2k+1} \times V \subset (M, \omega)$, where $V$ is orientable, spin, and assume that the minimal Maslov number of $L$ is $N_L=2k+2$.  Let $\sigma:= * \times [V] \in H_{n-(2k+1)}(L)$ denote the fundamental class of $V$ in $L$.  Let $\varphi\colon H_2(M, L) \to \mathbb{F}^\times$ be a representation such that the associated twisted quantum homology $QH^\varphi(L)$ vanishes, and suppose that we have
\begin{equation}
r_\varphi:= \sum_{A \in H_2(M, L)} GW_{0,2}^A(\sigma, pt)\varphi(A) \neq 0 \tag{$\star$},
\end{equation}
where $GW_{0,2}^A(x, y)$ denotes genus zero open Gromov Witten invariant with two boundary marked points, going through the cycles $x$ and $y$, in the homology class $A$ (see \S \ref{sec:openGW}).  The Maslov index of $A$ must be $2k+2$ and the assumptions imply that $L$ is $E^{1, \varphi}$-narrow.  Our second main result, stated again in Theorem \ref{thm:assumptionstar}, is then:
\begin{mainthm}
Let $L=S^{2k+1} \times V$ be a closed, monotone, orientable, and spin Lagrangian submanifold.  Let $\F$ be a field such that $\text{char } \F \nmid |\tor(L)|$, $\varphi\colon H_2(M, L;\Z) \to \F^\times$ a representation and $\mathcal{D}$ a generic pearl data triple.  Suppose further that assumption $(\star)$ holds.  Then $L$ is $E^{1, \varphi}$-narrow, the quantum Reidemeister torsion of $(L, \varphi)$ does not depend on $\mathcal{D}$ and $\tau(L, \varphi) = \dfrac{|\tor_{\text{ev}}(L)|}{|\tor_{\text{odd}}(L)|} r_\varphi^{-\chi(V)} \in \K(\F)$, where $|\tor_*(L)|$ is considered an element of $\F^\times$.
\end{mainthm}
\begin{cor}
Assumption $(\star)$ is satisfied for every narrow representation $\varphi$ whenever
\begin{itemize}
\item $L=S^1$.  Then $\tau(L, \varphi) = r_\varphi^{-1}$.
\item $L$ is a $m$-torus, $L=S^1 \times S^1\times \cdots \times S^1$ ($m$ times), where $m \geq 2$. Then $\tau(L, \varphi) = 1$
\item $L=S^1 \times \Sigma_g$, where $\Sigma_g$ is a closed orientable surface of genus $g \geq 1$.  Then $\tau(L, \varphi) = r_\varphi^{2(g-1)}$
\item $L$ is given by a product Lagrangian embedding $S^1 \times V \subset (S^2 \times X, \omega_{S^2}\oplus \omega_X)$, where $H_2(X, V) = 0$.
\end{itemize}
\end{cor}
We provide explicit examples in \S \ref{sec:ex} where the coefficient $r_\varphi$ is not trivial as an element of either $\K(\F)$ or $\F^\times / \pm \varphi(H_2(M, L))$.  The last case in the Corollary occurs for example if $V$ is closed and $X$ is the cotangent bundle of $V$.

\begin{rem}
The case $L=S^1\times S^2$ is not covered by our methods.  We provide an example due to Oakley-Usher \cite{Oa-Ush:certainlags} in \S \ref{sec:s1s2}.  In our subsequent paper \cite{Cha:3fold}, we study quantum Reidemeister torsion for narrow, orientable Lagrangian 3-manifolds and prove that $S^1\times S^2$ is indeed $E^{1, \varphi}$-narrow for every narrow representation.
\end{rem}

\subsection{Stably-free modules and the pearl complex}
Whitehead torsion is a more general version of Reidemeister torsion defined for acyclic chain complexes over a ring $R$, classically the group ring $\Z[\pi_1(X)]$, where $X$ is a finite $CW$-complex, see Milnor \cite{Mil:torsion}.  Its definition relies on the fact that the boundary modules of a bounded acyclic complex are stably-free $R$-modules.  As we point out in \S \ref{sec:whiteheadtorsion}, this fact is no longer true when considering general cyclic-graded chain complexes.  Therefore, in order to define Whitehead torsion for the pearl complex (as well as the Lagrangian Floer complex) one has to answer the following structural question:
\begin{qs}\label{qu:stablyfree}
Are the boundary modules in the pearl complex of a narrow Lagrangian stably-free?
\end{qs}
We discuss this in more details in \S \ref{sec:pearlboundary}, where we prove that over the group ring $\Z[H_2(M, L)]$, the pearl boundary modules are projective.

\subsection{Some questions related to Fukaya's symplectic s-cobordism conjecture}\label{sec:fukayascob}
In \cite{Fu:scobordism}, Fukaya sketches a construction of Whitehead torsion, for a pair of Lagrangian submanifolds (satisfying suitable assumptions that we omit), denoted by $\tau(L_1, L_2)$, which is defined precisely when the Floer homology $HF(L_1, L_2)$ of the pair vanishes.    He then states the following striking conjecture:
\begin{cnj}[Fukaya's symplectic s-cobordism conjecture]
Let $L_1, L_2$ be two Lagrangians submanifolds.  Assume that $HF(L_1, L_2)$ vanishes and moreover that $\tau(L_1, L_2)=1$.  Then, there exists a Hamiltonian isotopy $\phi$ such that $L_1 \bigcap \phi(L_2) = \emptyset$.
\end{cnj}
This leads us to the following questions:
\begin{qssnonum}
a.  Suppose that $L$ can be displaced by a Hamiltonian isotopy.  Does it follow that $\tau(L, \varphi) \equiv 1 \in \mathbb{F}^\times / \varphi(H_2(M, L))$ for every representation $\varphi\colon H_2(M, L) \to \mathbb{F}^\times$?  Does the converse hold if quantum Reidemeister torsion is replaced with quantum Whitehead torsion?\\
b. When $L_1=L_2$, do the above two notions of torsion coincide?  In other words, are the pearl and Floer complexes simple homotopy equivalent?
\end{qssnonum}
Given Theorem \ref{thm:first}, an affirmative answer to the first question would imply very strong restrictions on open Gromov-Witten invariants of displaceable Lagrangians.  The second question is motivated by the following classical phenomenon in algebraic topology:  there are topological spaces, namely three-dimensional lens spaces, which are homotopy equivalent (hence have the same homology) but that are not homeomorphic, because they do not have the same Reidemeister torsion, see \cite{Coh:simple}.

\section{The torsion of an acyclic chain complex}\label{sec:torsion}
\subsection{Bounded complexes}\label{sec:boundedcomp}
Following the presentation by Milnor \cite{Mil:torsion}, we recall the definition of torsion of an acyclic chain complex.  See also \cite{Coh:simple} and \cite{Tur:torsionbook}.  

The reduced Whitehead group of a field $\mathbb{F}$ is the multiplicative abelian group $$\K(\F) = \F^\times / \pm 1,$$ where $\F^\times = \F \backslash \{ 0 \}$. Given a finite-dimensional vector space $V$ over $\F$ of dimension $r$ and two bases $b=(b_1, \dots, b_r)$, $c=(c_1, \dots, c_r)$, there is a transition matrix $(a_{i, j})$, denoted simply $b/c$, expressing $b$ in terms of $c$:
$$b_i=\sum a_{ij} c_j.$$
Define $$[b/c] = \det b/c \in \K(\F).$$  This induces an equivalence relation on the set of bases of $V$, where $b=c$ if and only if $[b/c] = 1$.

Let $0 \to C_n \to C_{n-1} \to \cdots \to C_1 \to C_0 \to 0$ be a bounded chain complex over a field $\mathbb{F}$ such that each $C_i$ has a fixed finite basis $c_i=(c_{i,1}, \dots, c_{i, r_i})$.  Denote by $B_i$ the image of the boundary morphism $d\colon C_{i+1} \to C_i$ and by $Z_{i+1}$ its kernel.  

Choose bases $b_i=(b_{i,1}, \dots, b_{i,k_i})$ of $B_i$ and assume that $C_*$ is acyclic, so that $Z_{i} = B_i$ and the Euler characteristic is $\chi(C_*)=0$. One then has exact sequences
$$\xymatrix{0 \ar[r] & B_i \ar[r] & C_i \ar[r]^d & B_{i-1} \ar[r] & 0}$$
which split since we work over a field.  Given splittings $s_{i-1}\colon B_{i-1} \to C_i$, we get a new basis of $C_i$ obtained by concatenating the bases $s_{i-1}(b_{i-1})$ and $b_i$, which we write $s_{i-1}(b_{i-1})b_i$.  

Notice that given splittings $s_{i-1}$ and $s'_{i-1}$, we have $[s_{i-1}(b_{i-1})b_i/c_i]= [s'_{i-1}(b_{i-1})b_i/c_i]$.  Moreover, $s_{i-1}(b_{i-1})b_i$ and $s'_{i-1}(b_{i-1})b_i$ are equivalent bases.  We will often omit the section and write simply $b_{i-1}b_i$ for the new basis.

\begin{dfn}
The torsion of the chain complex $C_*$ with respect to the bases $\{ c_i \}$ is
$$\tau(C_*, c_* )= \prod_{i=0}^n [b_{i-1}b_i/c_i]^{(-1)^i} \in \K(\F).$$
\end{dfn}
As in \cite{Mil:torsion}, if we choose different bases $b'_i$ of $B_i$, we get
\begin{equation}\label{eq:changeb}
\prod_{i=0}^n [b'_{i-1}b'_i/c_i]^{(-1)^i}= \prod_{i=0}^n ([b_{i-1}b_i/c_i]*[b'_i/b_i]*[b'_{i-1}/b_{i-1}])^{(-1)^i}
\end{equation}
and the product $\prod_{i=0}^n ([b'_i/b_i]*[b'_{i-1}/b_{i-1}])^{(-1)^i}$ is equal to 1, so $\tau$ is independent of $b_i$.

Finally, torsion depends on the choice of basis $c_*$, but equivalent bases yield the same torsion.  Indeed, choosing another basis $c'_*$, we get
\begin{equation}\label{eq:tauchangebase}
\tau(C_*, c'_*) = \tau(C_*, c_*) \prod_{i=0}^n [c_i/c'_i]^{(-1)^i}
\end{equation}

\subsubsection{Non-acyclic complexes: Milnor's definition and torsion subgroups}\label{sec:milnorsdef}
If $C_*$ is not acyclic and $H_*(C)$ has a fixed basis $h_*$, then one has the following exact sequences:
\begin{align}\label{eq:exactseqmilnortorsion}
&\xymatrix{0 \ar[r] & Z_i \ar[r] & C_i \ar[r]^d & B_{i-1} \ar[r] & 0}\\
&\xymatrix{0 \ar[r] & B_i \ar[r] & Z_i \ar[r] & H_i \ar[r] & 0} \nonumber
\end{align}
which combine to yield a new basis $b_{i-1} h_i b_{i}$ of $C_i$.  
\begin{dfn}\label{eq:nonacyclictorsion}
The torsion is defined as
$$\tau(C_*, c_*, h_*) = \prod_{i=0}^n [h_i b_i b_{i-1}/c_i]^{(-1)^i} \in \K(\F).$$
Choosing equivalent bases to $c_*$ or $h_*$ does not affect its value.
\end{dfn}
Suppose now that $C_*$ is a bounded chain complex that is finitely generated as a $\Z$-module.  As above, let $B_i$ and $Z_i$ be the groups of boundaries and cycles, which are free $\Z$-modules, say $B_i \cong \Z^{k(i)}$ and $Z_i \cong \Z^{r(i)}$.  Write $H_i(C_*;\Z) = \Z^{r(i) - k(i)}\oplus \Z/ a_1 \Z \oplus \cdots \oplus \Z/ a_{s(i)} \Z$.  By standard algebra for modules over principal ideal domains, one can choose bases of $Z_i$ and $B_i$ such that, in the second exact sequence of (\ref{eq:exactseqmilnortorsion}), we have
\begin{align}\label{eq:trivialext}
\Z^{k(i)}=B_i & \to Z_i=\Z^{r(i)} \nonumber\\ 
b_l & \mapsto 
\begin{cases}
a_l z_l \quad 1 \leq l \leq s(i)\\
z_l     \quad s(i) < l \leq k(i) 
\end{cases}
\end{align}
In other words, there is only one free extension (of a given rank) of $H_i(C_*;\Z)$ by another free module (of a given rank), and, in the appropriate bases, it is given by the obvious maps written above.  
\begin{rem}\label{rem:changebasisinZ}
Any two basis of a free $\Z$-module are equivalent, since the transition matrix has determinant plus or minus one.
\end{rem}

Now, take a field $\F$ whose characteristic does not divide any of the $a_l$'s and tensor both sequences in (\ref{eq:exactseqmilnortorsion}) with $\F$.  This preserves exactness, boundaries, and cycles, by assumption on the characteristic.  Note also that a $\Z$-basis $h_*$ of $H_*(C_*;\Z)/ \tor$, called a basis of the free part of homology, gives a basis of $H_*(C_*;\F)$, denoted by $h_* \otimes_\mathbb{F} 1$ or simply $h_*$ when the context is clear.  By the remark above, all bases of $H_*(C_*; \F)$ obtained this way are equivalent.

Then, the second sequence above (considered over $\F$) is now free and is still given by
(\ref{eq:trivialext}).  Therefore, we get
$$[(h_i\otimes_\mathbb{F} 1) b_i b_{i-1}/c_i] = \prod_{j=1}^{s(i)} a_j = |\tor H_i(C_*;\Z)| \in \K(\F).$$
Finally, 
\begin{equation}\label{eq:torsioneqtorsion}
\tau(C_* \otimes \F, c_* \otimes 1, h_* \otimes 1) = \prod_i |\tor(H_i(C_*;\Z))|^{-1^i}
\end{equation}
Simply put, torsion equals torsion!  In case $C_*$ is the cellular or Morse complex of a manifold $X$, we will often abbreviate this formula as
$$\tau(C_* \otimes \F, c_* \otimes 1, h_* \otimes 1) = \dfrac{|\tor_{\text{ev}}(X)|}{|\tor_{\text{odd}}(X)|}$$

\begin{rem}\label{rem:basis_h_canonical}
Combining Remark \ref{rem:changebasisinZ} and formula (\ref{eq:tauchangebase}), we see that this torsion depends only on the equivalence class of $c_*$ and the unique class of $h_* \otimes 1$.
\end{rem}
\subsection{Periodic complexes}\label{sec:periodictorsion}
Consider now a based 2-periodic chain complex $C_{[*]}$ over a field $\mathbb{F}$, with bases $c_{[i]}$:
$$\xymatrix{C_{[1]} \ar@/^1pc/[r]^{d} & \ar@/^1pc/[l]^d C_{[0]}}$$
The primary example we have in mind is the pearl complex with a cyclic grading, see \S \ref{sec:pearltorsion}.  A similar version using maximal abelian torsion for the Floer complex has been considered by Lee \cite[\S 2.2.3]{Lee:torsion1}.

Given such an acyclic chain complex, we may pick bases $b_{[i]}$ and define its torsion:
\begin{dfn}
The torsion of a 2-periodic acyclic chain complex is
$$\tau_2(C_{[*]}, c_{[*]} ) = \dfrac{[b_{[1]}b_{[0]}/c_{[0]}]}{[b_{[0]}b_{[1]}/c_{[1]}]} \in \K(\F).$$
Equivalent bases of $C_{[*]}$ yield the same torsion.
\end{dfn}
Equation (\ref{eq:changeb}) still applies to prove that $\tau_2$ is independent of the choice of bases $b_{[i]}$.  Choosing different sections again leaves $\tau_2$ invariant.

This torsion generalizes the one defined in the previous section.  Given a bounded chain complex $C_*$, define a 2-periodic complex by setting 
$$(C_{[0]}, c_{[0]}) = \bigoplus_{k \text{ even}} (C_k, c_k), \quad (C_{[1]}, c_{[1]}) = \bigoplus_{k \text{ odd}} (C_k, c_k) $$
with the differential being simply the direct sum of the differentials of $C_*$.   Note that $\chi(C_{[*]}) = \chi(C_*)$. Obviously, $B_{[0]} = \oplus_{k \text{ even}} B_k$ and $B_{[1]} = \oplus_{k \text{ odd}} B_k$.  The split sequences
$$
\xymatrix{
0 \ar[r] & B_i \ar[r] & C_i \ar[r] & B_{i-1} \ar@/^1pc/[l]^{s_{i-1}} \ar[r] & 0
}
$$
can be added up by defining $s_{[0]}=\oplus s_{\text{even}}$, $s_{[1]}=\oplus s_{\text{odd}}$, to yield two splittings
\begin{equation}\label{eq:split}
\xymatrix{
0 \ar[r] & B_{[i]} \ar[r] & C_{[i]} \ar[r] & B_{[i-1]} \ar@/^1pc/[l]^{s_{[i-1]}} \ar[r] & 0
}
\end{equation}
After a reordering of the bases, we get block diagonal matrices $b_{[1]}b_{[0]}/c_{[0]} = \oplus_{k \text{ even}} b_{k+1}b_{k}/c_{k}$ and $b_{[0]}b_{[1]}/c_{[1]}=\oplus_{k \text{ odd}} b_{k+1}b_{k}/c_{k}$, therefore
$$\tau(C_*, c_*) = \tau_2(C_{[*]}, c_{[*]}).$$

\begin{remnonum}
If an acyclic complex has a $\Z / 2k \Z$-grading, then one can cook up a $\Z /2\Z$-graded acyclic complex out of it, by reducing the grading modulo 2, just as above, hence there is no need to define the notion for such complexes, 2-periodic complexes are enough for this purpose.  It is not clear how to adapt these definitions to $\Z / (2k+1)\Z$-graded complexes.
\end{remnonum}
\subsection{Stably-free modules and Whitehead torsion}\label{sec:whiteheadtorsion}
It is possible to define the torsion of free, bounded acyclic complexes with a preferred basis, over rings $R$ satisfying the invariant basis property (IBP) - e.g. commutative rings and group rings - as an element of the Whitehead group of that ring, denoted by $K_1(R)$.  This is ultimately possible because the boundary modules $B_i$ are stably-free (meaning that there exist positive integers $s_i$ and $k_i$ such that $B_i \oplus R^{s_i} \cong R^{k_i}$),  by a simple induction argument.  See e.g.\ Cohen \cite[\S 13]{Coh:simple} or \cite[\S 4]{Mil:torsion} for more on this.

Unfortunately, the same procedure does not work for periodic complexes over such rings, since the boundary modules $B_{[i]}$ are not automatically stably-free, in fact not even projective, as the following examples show.  Moreover, it is easy to see that $B_{[0]}$ is stably-free if and only if $B_{[1]}$ is.

\newcounter{Excounter}
\stepcounter{Excounter}
\begin{exsnonum}
(\arabic{Excounter}) \stepcounter{Excounter}
Let $B_{[0]}$ be a projective module that is not stably-free over some ring $R$ satisfying the (IBP) property (these exist), so that there exists a complement $B_{[1]}$ to $B_{[0]}$ in a free module, i.e.\ $B_{[0]}\oplus B_{[1]} \cong R^k,$ for some $k \geq 1$.  Set $C_{[0]} = C_{[1]}=R^k$ and define differentials by projecting to each factor
\begin{align*}
d\colon C_{[1]} & \to C_{[0]} & \delta\colon C_{[0]} &\to C_{[1]}\\
(b_0, b_1) & \mapsto (b_0, 0) & (b_0, b_1) &\mapsto (0, b_1)
\end{align*}
Then $C_{[*]}$ is acyclic but the boundary modules are not stably-free, since  $\text{im } d \cong B_{[0]}$ and $\text{im } \delta \cong B_{[1]}$.\\

\noindent(\arabic{Excounter})\stepcounter{Excounter}
Take $R=\Z[\Z/p\Z] \cong \Z[t]/(t^p -1)$, the group ring of $\Z/p\Z$, for $p$ an odd prime number.  The polynomials $t-1$ and $1+t+\dots +t^{p-1}$ are zero divisors in this ring.  Set $C_{[0]} = C_{[1]} =R$ with differentials
\begin{align*}
d\colon C_{[1]} & \to C_{[0]} & \delta\colon C_{[0]} &\to C_{[1]}\\
r & \mapsto r(t-1) & r &\mapsto r(1+t+\dots +t^{p-1})
\end{align*}
This complex is acyclic.  The boundaries are $R$-submodules of $R$ given by $B_{[0]} = (t-1)R$, $B_{[1]} = (1+t+\dots +t^{p-1})R$.  These modules are not projective, hence not stably-free, since projective submodules of $R$ correspond to idempotents elements of $R$, and the only idempotent elements of $\Z[G]$, with $|G|$ finite, are $0$ and $1$, see e.g.\ Weibel \cite[Chapter 1, Example 2.1.2 and Chapter 2, \S 2, Corollary 2.5.3]{Weibel:kbook}.
\end{exsnonum}


\section{The pearl complex and its torsion}\label{sec:pearlcomplex}
We refer to Biran--Cornea's papers \cite{Bi-Co:Yasha-fest, Bi-Co:rigidity, Bi-Co:lagtop} for foundations and applications of Lagrangian quantum homology.  The version we use here (with oriented moduli spaces of pearls) is adapted from \cite{Bi-Co:lagtop}.
\subsection{Setting}
Throughout the text, $(M, \omega)$ is a $2n$-dimensional symplectic manifold that is connected and convex at infinity whenever it is not closed.  The space of $\omega$-compatible almost complex structures on $M$ is denoted by $\mathcal{J}_{\omega}$.

All Lagrangian submanifolds $L \subset (M, \omega)$ are closed and connected.  Moreover, they are endowed with a fixed choice of orientation and spin structure which we do not write.

Let $\omega\colon \pi_2(M, L) \to \R$ be given by the symplectic area of discs and  $\mu\colon \pi_2(M, L) \to \Z$ denote the Maslov index.  The generator of the image of $\mu$ is called the minimal Maslov index and is denoted by $N_L$.  Lagrangians are assumed \textbf{monotone}, that is, there exists a constant $\eta> 0$ such that
\begin{itemize}
	\item $\omega = \eta\mu$
	\item $N_L \geq 2$
\end{itemize}
Since $L$ is orientable, $N_L$ is even and the second condition above follows from the first.

\subsection{The 2-periodic pearl complex}\label{sec:2pearl}
Fix a triple $\mathcal{D}=(f, \rho, J)$, where $f\colon L \to \R$ a Morse function, $\rho$ is a Riemannian metric such that $(f, \rho)$ is a Morse-Smale pair, and $J \in \mathcal{J}_{\omega}$ is an almost complex structure compatible with $\omega$.  

Set $$\mathcal{C}_k = \mathcal{C}_k(\mathcal{D}) = \Z[H_2(M, L;\Z)] \langle \text{Crit}_{k}f \rangle, \quad k=0, \dots, n=\dim L,$$ where $\text{Crit}_k f$ is the set of critical points with Morse index $k$ and $\Z[G]$ is the group ring of a group $G$.  We write the Morse index of a critical point $x$ as $|x|$.

For a generic triple $\mathcal{D}$, the pearl differential is defined by
\begin{align*}
d\colon \oplus_k \mathcal{C}_{k}(\mathcal{D}) & \to \oplus_k \mathcal{C}_{k}(\mathcal{D})\\
\text{Crit f} \ni x & \mapsto \sum_{y \in \text{Crit } f}\big(\sum_{\substack{A \in H_2(M, L)\\|x|-|y|-1+\mu(A)=0}} \#(\mathcal{P}(x,y,A))A\big)y
\end{align*}
where $\#(\mathcal{P}(x,y,A))$ is the (signed) number of pearls in the homology class $A$ going from $x$ to $y$.  When $\mu(A)=0$, a pearl is simply a negative gradient flow line of $f$.  This morphism decomposes as a finite sum 
\begin{equation}\label{eq:dsum}
d = d_M + d_1 + \dots
\end{equation}
 where $d_M\colon \mathcal{C}_k \to \mathcal{C}_{k-1}$ is the Morse differential and $d_i\colon \mathcal{C}_{k} \to \mathcal{C}_{k-1 +iN_L}$ counts pearls of Maslov index $iN_L$.  Note that $d_i$'s are not differentials, they need not square to zero, even though $d^2=0$. 

Since $N_L$ is even, $k$ and $k-1 +i N_L$ have different parity, hence there is a 2-periodic pearl complex over $\Z[H_2(M, L)]$, defined by
$$\mathcal{C}_{[*]}(\mathcal{D})=\bigoplus_{k \equiv [*] \; \text{mod } 2} \mathcal{C}_k(\mathcal{D}), \quad [*]=0,1$$
with an induced differential $d\colon \mathcal{C}_{[*]} \to \mathcal{C}_{[*-1]}$.

The homology of this complex is called the quantum homology of $L$, denoted by $QH_{[*]}(L)$, or simply $QH(L)$.  It is independent of generic choices of $\mathcal{D}$.  If $QH(L) =0$, we say that $L$ is \textit{narrow}.

%

\subsection{Narrow representations and torsion}\label{sec:pearltorsion}
Fix a field $\mathbb{F}$ and a representation $\varphi\colon H_2(M, L) \to \mathbb{F}^\times = GL(1, \mathbb{F})$.  This induces a ring morphism (by convention, ring morphisms map 1 to 1)
$$\varphi\colon \Z[H_2(M, L)] \to \mathbb{F},$$
so that $\mathbb{F}$ becomes a $\Z[H_2(M, L)]$-module.  This defines a 2-periodic chain complex over $\mathbb{F}$ by setting
$$\mathcal{C}^{\varphi}_{[*]}(\mathcal{D})=\mathcal{C}_{[*]}(\mathcal{D})\otimes_{\Z[H_2(M, L)]} \mathbb{F}, \quad d^{\varphi} = d\otimes 1.$$
As above, the homology of this new complex, denoted by $QH^{\varphi}(L)$, does not depend on $\mathcal{D}$.  If it vanishes, we say that $L$ is $\varphi$-narrow, which implies $\chi(L;\F)=0$.
\begin{lem}\label{lem:narrow_implies_fnarrow}
If $QH(L) = 0$, then $QH^\varphi(L)=0$ for every representation $\varphi\colon H_2(M, L) \to \mathbb{F}^\times$.
\end{lem}
\begin{proof}
The complexes $\mathcal{C}_{[*]}(\mathcal{D})$ are free $\Z[H_2(M, L)]$-modules.  Since $QH(L)=0$, the boundaries $d(\mathcal{C}_{[*]}(\mathcal{D}))$ are projective $\Z[H_2(M, L)]$-modules, by Proposition \ref{prop:projective}.  By the universal coefficient theorem (see e.g.\ Weibel \cite[Theorem 3.6.1]{Weibel:book-hom-alg}), there is an exact sequence
$$\xymatrix{
0 \ar[r] & QH_{[*]}(L)\otimes_{\Z[H_2(M, L)]} \mathbb{F} \ar[r] & QH_{[*]}^\varphi(L) \ar[r] & \tor^{\Z[H_2(M, L)]}(QH_{[*-1]}(L), \mathbb{F}) \ar[r] & 0
}$$
The second and fourth term in this sequence are null, therefore $L$ is $\varphi$-narrow.
\end{proof}

The set of narrow representations of $L$ over $\mathbb{F}$ is defined by
$$\mathcal{N}(L, \mathbb{F})= \left\{ \varphi\colon H_2(M, L) \to \mathbb{F}^\times \; | \; L \text{ is } \varphi\text{-narrow} \right\}.$$
We also define
$$\mathcal{N}^{\text{free}}(L, \mathbb{F}) = \left\{ \varphi\colon H_2(M, L)/\tor \to \mathbb{F}^\times \; | \; L \text{ is } \varphi\text{-narrow} \right\}.$$
Picking a basis $z_1, \dots, z_{b_2(M, L)}$ of $H_2(M, L;\Z)/\tor$, we consider $\mathcal{N}^{\text{free}}(L, \mathbb{F})$ as a subset of $(\F^\times)^{b_2(M, L)}$.  It is in fact an open subset, by arguments of Biran-Cornea \cite[\S 3.1]{Bi-Co:lagtop}.

Given $\varphi \in \mathcal{N}(L, \mathbb{F})$ and $\mathcal{D}$ a generic set of data, there is a preferred basis for $\mathcal{C}^{\varphi}_{[*]}(\mathcal{D})$ given by $\text{Crit}_{[*]} f$.  Proceeding as in \S \ref{sec:periodictorsion}, we have:
\begin{dfn}\label{dfn:torsion}
The quantum Reidemeister torsion of the pair $(L, \varphi)$ is
$$\tau((L, \varphi), \mathcal{D}) = \tau_2(\mathcal{C}^{\varphi}_{[*]}(\mathcal{D}), \text{Crit}_{[*]} f) \in \K(\F).$$
It induces a function $\tau(L, \cdot, \mathcal{D}) \colon \mathcal{N}^{\text{free}}(L, \mathbb{F}) \to \K(\F).$
\end{dfn}

\begin{rems}\label{rem:nozgrading}
a.  Narrow representations do not always exist (for example, take any $L$ with non-vanishing Euler characteristic over $\F$). However, when $N_L=2$, narrow representations can be detected by computing partial derivatives of the Landau-Ginzburg superpotential, see Biran--Cornea \cite[\S 3.3]{Bi-Co:lagtop} for more on this.\\
\noindent b. It is possible to give a $\Z$-grading to the pearl complex by introducing a Novikov variable that keeps track of Maslov indices.  However, doing this makes the pearl complex unbounded, and the differential becomes periodic.  Thus, to define torsion in this context, one needs to take an infinite product of determinants which repeat themselves every multiple of $N_L$.  To avoid this type of issue, we chose to get rid of the Novikov  variable altogether and use a $\Z / 2\Z$-grading.
\end{rems}

\subsection{Invariance}
In this paper, we prove that $\tau((L, \varphi), \mathcal{D})$ is independent of $\mathcal{D}$, whenever $(L, \varphi)$ belongs to a certain class that we call $E^{1, \varphi}$-narrow and which contains tori endowed with a narrow representation, see \S \ref{sec:e1narrow}.  
\begin{remnonum}
There are other contexts in which the behaviour of torsion under changes of data has been studied.  See for example Hutchings \cite{Hut:torsion}, Hutchings--Lee \cite{HutLee:torsion} and Lee \cite{Lee:torsion1, Lee:torsion2} for finite/infinite-dimensional Morse-Novikov complex of circle-valued Morse functions.  In their context, torsion is not invariant, however its product with a zeta function encoding gradient periodic orbits is invariant.

Suarez \cite{Su:exactcob} studies torsion for exact Lagrangian cobordisms, Sullivan \cite{SullM:kinvariants} a version of the Lagrangian Floer complex, and Abouzaid-Kragh \cite{Abou-Kragh:simplehomtypenearby} study torsion for exact Lagrangians.  In these contexts, torsion is invariant.  Abouzaid-Kragh use the action filtration to prove this; our Theorem \ref{thm:specseq} is a generalization  of their methods to the non exact setting.
\end{remnonum}

\subsection{Oh's spectral sequence and a family of chain complexes associated to $d_1$}\label{sec:d1complex}
In this section, we will briefly need Novikov ring coefficients in order to define Oh's spectral sequence \cite{Oh:spectral} in Lagrangian Floer homology.  In the pearl context, this corresponds to the degree spectral sequence of Biran--Cornea \cite{Bi-Co:rigidity}.  We follow the presentation by Biran--Membrez \cite[Appendix A]{Bi-Me:cubic}.

Set $\Lambda = \Lambda_R = R[t, t^{-1}]$ the ring of Laurent polynomials in $t$.  We set $\deg t = |t|=- N_L$.  Here, $R$ is a ring, which will be either $\Z[H_2(M, L)]$, $\Z$ or a field $\mathbb{F}$. Set also 
$$P_i =
\begin{cases}
R t^{-i/N_L} \quad i \equiv 0 \mod N_L\\
0 \quad \text{otherwise}
\end{cases}
$$
Given a generic pearl triple $\mathcal{D} = (f, \rho, J)$, recall from formula (\ref{eq:dsum}) that $d = d_M + d_1 + d_2 + \dots$.  Squaring this, we get $0 = d_M^2 + d_M d_1 + d_1 d_M + d_M d_2 + d_1^2 + d_2 d_M + \dots$.  From this, we see that $d_1$ induces a map in Morse homology and this map squares to zero.  Expanding on these ideas, Oh obtained a spectral sequence in Floer homology, which in our context gives the following
\begin{thm}[The degree spectral sequence]\label{thm:specseq}
There is a spectral sequence of algebras $\{ E_{p,q}^r, d^r; \Lambda \}$, with $d^r$ of bidegree $(-r, r-1)$, called the degree spectral sequence, having the following properties:
\begin{itemize}
\item $E_{p,q}^0 = \Ccal_{p+q-pN_L}(\mathcal{D}) \otimes P_{pN_L}$, $d^0 = d_M \otimes 1$
\item $E_{p,q}^1 = H_{p+q-pN_L}(L; R)\otimes P_{pN_L}$, $d^1 = d_{1*}\otimes t$, where
$$d_{1*} \colon H_r(L;R) \to H_{r-1+N_L}(L;R)$$
is induced from $d_1$.
\item $d_{1*}$ satisfies the Leibniz rule with respect to the Morse intersection product, i.e.\ $d_{1*}(x \cdot y) = d_{1*}(x)\cdot y +(-1)^{n-|x|}x \cdot d_{1*}(y)$
\item As the differential has bidegree $(-r, r-1)$, $\{ E_{p,q}^r, d^r\}$ collapses after at most $\lfloor \frac{n+1}{N_L}\rfloor$ pages.  Moreover, it converges to $QH(L; \Lambda)$.  In particular, when $R$ is a field, we have
$$\oplus_{p+q=l}E_{p,q}^\infty \cong QH_l(L;\Lambda)$$
\end{itemize}
\end{thm}
\begin{remnonum}
The Leibniz property of $d_{1*}$ was first proven by Buhosvky \cite{Bu:toriaudin}.
\end{remnonum}
We will sometimes indicate the Morse function and the triple $\mathcal{D}$ when writing Morse homology, e.g.\ $d^{\mathcal{D}}_{1*}\colon H_r^f(L; R) \to H_{r+N_L-1}^f(L; R)$.  From the theorem, we get $d_{1*}^2 =0$, thus for each $0 \leq k \leq N_L -2$, there is a chain complex of (not necessarily free!) $R$-modules
\begin{equation}\label{eq:d1complex}
\xymatrix{0 \ar[r] & H^f_k(L) \ar[r] & H^f_{k-1+N_L}(L) \ar[r] & \cdots \ar[r] & H_{k + n(N_L-1)}^f(L) \ar[r] & 0}
\end{equation}
We write this family of complexes as $(H_*(L), d_{1*})$.  Notice the particular case $N_L=2$, as there is a single complex:
\begin{equation*}
\xymatrix{0 \ar[r] & H_0(L) \ar[r]^{d_{1*}} & H_{1}(L) \ar[r] & \cdots \ar[r] & H_{n-1}(L) \ar[r]^{d_{1*}} & H_n(L) \ar[r] & 0}
\end{equation*}

We list below a few properties of this chain complex, which will be useful when computing torsion later on.  They all follow from the general machinery of \cite{Bi-Co:rigidity} and Theorem \ref{thm:specseq}.
\begin{itemize}
\item[(Naturality)] Given any ring morphism $\varphi\colon \Z[H_2(M, L)] \to \mathbb{F}$, $(E_{p,q}^r, d^r; \Lambda)$ is natural with respect to the associated change of coefficients.  The resulting spectral sequence is denoted by $(E_{*,*}^{r, \varphi}, d^{r, \varphi}; \mathbb{F})$.
\item[(Independence)]\label{indep} $d_{1*}$ is independent of $\mathcal{D}$.  Namely, first recall that given $\mathcal{D}_1=(f, \rho, J)$ and $\mathcal{D}_2=(g, \rho', J')$ two generic pearl data triples, there is a canonical comparison morphism $\Phi^{\mathcal{D}_1}_{\mathcal{D}_2}$ between the associated pearl complexes which induces a morphism of spectral sequence.  Moreover, this morphism induces, on the first page, yet another map which coincides with the homology value of the usual canonical Morse comparison morphisms (see Schwarz \cite{Sc:Morse-homology} for relevant definitions) $[\Phi^f_g]_* \colon H_*^f(L) \to H_*^g(L)$ between the respective Morse homologies of $f$ and $g$, where we omit Riemannian metrics from the notation, and makes the following diagram commute (over $\Z$ or for any field representation of $H_2(M, L)$):
\begin{equation}\label{eq:d1commuteMorse}
\xymatrix{ H_*^f(L) \ar[rr]^{d^{\mathcal{D}_1}_{1*}} \ar[d]_{[\Phi^f_g]_*} & & H_{*-1+N_L}^f(L) \ar[d]^{[\Phi^f_g]_*}\\
H_*^g(L) \ar[rr]^{d_{1*}^{\mathcal{D}_2}} & & H_{*-1+N_L}^g(L)}
\end{equation}
\end{itemize}

Using the spectral sequence, one gets:
\begin{thm}[\cite{Bu:toriaudin}, \cite{Bu:toripreprint} (Proof of Theorem 1)]\label{thm:buhovsky}
Let $L^n$ be a monotone Lagrangian submanifold whose Morse homology ring is generated by $H_{n-1}(L; \F)$ and assume that $L$ is $\varphi$-narrow for some representation $\varphi\colon H_2(M, L) \to \mathbb{F}^\times$.  Then $N_L =2$ and the chain complex $(H_*(L; \mathbb{F}), d^{\varphi}_{1*})$ is acyclic.  In other words, $0 = E_{*,*}^{2, \varphi} = E_{*,*}^{3, \varphi} = \dots = E_{*,*}^{\infty, \varphi}$.
\end{thm}
\begin{rem}
Lagrangians with a homology ring as above are for example tori or $S^1 \times \Sigma_g$.  See Corollary \ref{cor:assumptionstar}.
\end{rem}

\subsection{Algebraic structure of the pearl boundary modules}\label{sec:pearlboundary}
In view of the discussion on stably-free modules in \S \ref{sec:whiteheadtorsion}, we were asking in Question \ref{qu:stablyfree} whether the pearl boundary modules are stably free.  As of now, we cannot answer this question positively, neither can we find a counterexample. Note that the counterexamples of \S \ref{sec:whiteheadtorsion} were purely algebraic and did not make use of geometry.  Nevertheless, we can make a first step in the right direction:
\begin{prop}\label{prop:projective}
Assume that $L$ is narrow, $QH(L) = 0$.  Then, the pearl boundary modules $b_{[i]}$ are projective $\Z[H_2(M, L)]$-modules.
\end{prop}
\begin{proof}
The proof uses the quantum product over the ring $\Z[H_2(M, L)]$ (see Biran--Cornea \cite{Bi-Co:rigidity}).  Recall that given any generic pearl triple $\mathcal{D}=(f, \rho, J)$, there is another generic triple $\mathcal{D}'=(g, \rho', J')$ with the following properties:
\begin{itemize}
\item The Morse function $g$ has a single maximum, denoted by $L_g$.
\item There is a quantum product (over the ring $\Z[H_2(M, L)]$)
$$\circ\colon \Ccal_{[k]}(g, \rho', J') \otimes_{\Z[H_2(M, L)]} \Ccal_{[l]}(f, \rho, J) \to \Ccal_{[k+l-n]}(f, \rho, J)$$
which endows Lagrangian quantum homology with a unital ring structure and satisfies the Leibniz rule with respect to the pearl differentials.  Moreover, $L_g$ is a \textit{chain-level} unit for this product, i.e.\ $L_g \circ x = x, \; \forall x \in \Ccal(f, \rho, J)$.
\end{itemize}
We denote by $d_f, d_g$ the pearl differentials of each pearl complex.  Note that $d_g(L_g) = 0 $ for degree reasons (see (\ref{eq:dsum})).  Since $L$ is narrow, this means that there exists $\sigma \in \Ccal_{[n-1]}(g)$ such that $d_g(\sigma) = L_g$.  Assume now that $y=d_f(x)$, then $d_f(\sigma \circ y) = d_g(\sigma)\circ y +(-1)^{n-|\sigma|}\sigma\circ d_f(y) = L_g \circ y  + (-1)^{n-|\sigma|}\sigma\circ d_f^2(x)= y$.  Therefore, left multiplication with $\sigma$ provides a splitting of $d_f$, which implies that $\Ccal(f) \cong \text{Im } d \oplus \ker d$, hence the boundary modules are projective.
\end{proof}

\section{Open Gromow-Witten invariants and the differential $d_{1*}^\varphi$}\label{sec:openGW}
In this section, we introduce the notion of open Gromov-Witten invariants that appears in Theorem \ref{thm:first}.  Contrary to absolute Gromov-Witten invariants, defined using holomorphic spheres (see e.g.\ McDuff-Salamon \cite{McD-Sa:Jhol-2}), care must usually be taken in the open case, as bubbling of discs is a codimension one phenomenon.  In our context however, this does not happen and the definition is rather straightforward.  For a broader view of the subject, where bubbling issues are present, the reader is referred to Biran-Cornea \cite{Bi-Co:lagtop}, Georgieva-Zinger \cite{Georg-Zinger:realGWtheorygenera}, Liu \cite{Liu:openGW}, Solomon \cite{Solomon:intersection, Solomon:thesis}, Solomon-Tukachinsky \cite{Sol-Tuk:openGWbounding}, Welschinger \cite{Welsch:discs4dim}, etc.\

Fix $A \in H_2(M, L)$ such that $\mu(A) = N_L$, $J$ an almost complex structure compatible with $\omega$ and denote by $\mathcal{M}(A;J) = \{ u\colon (D^2, S^1) \to (M, L) \; | \; \bar{\partial}_J u = 0, \; [u] = A \}$ the space of $J$-holomorphic discs in the homology class $A$.  By the work of Lazzarini \cite{Laz:decomp}, these are automatically simple curves, as $\mu(A)  = N_L$ and $L$ is monotone.  Therefore, for generic $J$, $\mathcal{M}(A;J)$ is a smooth manifold of dimension $\dim L+\mu(A) = n+\mu(A)$.  Our implicit choice of spin structure on $L$ induces an orientation on this space, by the work of Fukaya-Oh-Ohta-Ono \cite[Chapter 8]{FO3:book-vol2}.  Moreover, by monotonicity and Gromov compactness for discs (see Frauenfelder \cite{Fr:msc}), $\mathcal{M}(A;J)$ is a closed manifold and its bordism class does not depend on $J$.  In other words, given two regular $J_0, J_1$ and a generic path $\mathbf{J} = \{ J_t, \; t \in [0,1] \}$ between them, there is a smooth, oriented, closed cobordism between  $\mathcal{M}(A;J_0)$ and $\mathcal{M}(A;J_1)$, given by $\bigcup_{t \in [0,1]} \mathcal{M}(A;J_t)$.  

There is an evaluation map
\begin{align*}
ev_2\colon \mathcal{M}(A;J) \times ((S^1)^2 \backslash \Delta) / \mathcal{G} & \to L \times L\\
(u, \theta_1, \theta_2)&  \mapsto (u(\theta_1), u(\theta_2))
\end{align*}
where $\Delta$ is the diagonal and $\mathcal{G}$ is the group of biholomorphism of the disc acting via the formula $g \cdot (u, \theta_1, \theta_2)= (u \circ g, g^{-1}(\theta_1), g^{-1}(\theta_2))$.
Note that $\mathcal{M}(A;J) \times ((S^1)^2 \backslash \mathbf{\Delta})$ is a non compact manifold of dimension $n +\mu(A) +2  = n + 2 + N_L$. It represents the space of $J$-holomorphic discs in class $A$ with two distinct marked points.

Now, let $\mathcal{D}=(f, \rho, J)$ be a generic pearl data triple.  Denote by  $\psi_t$ the negative gradient flow of $f$ with respect to $\rho$.  Let $x \in \Crit f$, and define $W^u(x) = \{ p \in L \; | \; \lim_{t \to -\infty} \psi_t(p) = x\}$ the unstable manifold of $x$; it is a non proper submanifold of $L$.  Similarly, $W^s(x) = \{ p \in L \; | \; \lim_{t \to \infty} \psi_t(p) = x\}$ is the stable manifold of $x$.

Fix $x, y \in \Crit f$ such that $|y| = |x|-1 + N_L$.  By genericity of $\mathcal{D}$, the map $ev_2$ is transverse to $W^u(x) \times W^s(y)$.  Therefore,
\begin{align*}
\mathcal{P}(x, y, A; \mathcal{D}) = ev_2^{-1}(W^u(x) \times W^s(y))
\end{align*}
is a zero dimensional manifold; it is the set of pearls from $x$ to $y$ in the homology class $A$.  By the results of Biran-Cornea \cite[\S 3.4]{Bi-Co:rigidity}, \cite[Appendix A]{Bi-Co:lagtop}, it is in fact a compact orientable manifold, hence there is a well-defined integer $\# \{ \mathcal{P}(x, y, A; \mathcal{D}) \}$.  Recall the decomposition $d = d_M + \sum_i d_i$ in (\ref{eq:dsum}).  Then, over $\Z[H_2(M, L)]$,
\begin{equation}\label{eq:d1phi}
d_1(x) = \sum_y (\sum_{\mu(A)=N_L} \# \{\mathcal{P}(x, y, A; \mathcal{D})\} A) y 
\end{equation}
By definition, $\mathcal{P}(x, y, A; \mathcal{D})$ also represents the set of $J$-holomorphic discs in homology class $A$, with two distinct boundary marked points intersecting respectively $W^u(x)$ and $W^s(y)$.

Fix $x,y \in H^f_*(L; \Z)$ two Morse homology classes in the free part of $H_*^f(L;\Z)$, such that $|y| = |x|-1 + N_L$.  Write $x = \sum a_i x_i$, $y = \sum b_i y_i$, with $a_i, b_i \in \Z$, $x_i, y_i \in \Crit f$.
\begin{dfn}
Let $L$ be a closed, orientable, spin, monotone Lagrangian submanifold of $(M, \omega)$ and $A \in H_2(M, L; \Z)$ be such that $\mu(A) = N_L$. Fix $\mathcal{D} = (f, \rho, J)$ a generic pearl data triple and $x, y \in H^f(L;\Z)$ satisfying $|y| = |x|-1 + N_L$.  The number $$GW^A_{0,2}(x, y) = \sum_{i, j} a_i b_j \#\{ \mathcal{P}(x_i, y_j, A; \mathcal{D}) \} \in \Z$$
is called a genus zero open Gromov-Witten invariant in class $A$, intersecting the cycles $x$ and $y$.
\end{dfn}

Let $\varphi\colon H_2(M, L;\Z) \to \F^\times$ be a representation.  By formula (\ref{eq:d1phi}) above, we get the following interpretation of $d_{1*}$ and $d_{1*}^\varphi$ from \S \ref{sec:d1complex}:
\begin{prop}\label{prop:d1openGW}
The differentials on the first page of Oh's spectral sequences $E^1, E^{1, \varphi}$ are expressed in terms of open Gromov-Witten invariants of minimal Maslov index, with two marked points.  Given $x \in H_{|x|}(L;\Z)$ or $x \in H_{|x|}(L;\F)$, then
$$d_{1*}(x) = \sum_{y \in H_{|x| + N_L - 1}} \left(\sum_{\mu(A)=N_L} GW_{0,2}^A(x, y)A \right) y$$
$$d_{1*}^\varphi(x) = \sum_{y \in H_{|x| + N_L - 1}} \left(\sum_{\mu(A)=N_L} GW_{0,2}^A(x, y)\varphi(A)\right) y$$
\end{prop}
Using the commutative diagram (\ref{eq:d1commuteMorse}), which ultimately relies on the compactness results of \cite[\S 3.4]{Bi-Co:rigidity}, one obtains:
\begin{cor}
$GW^A_{0,2}(x, y)$ does not depend on the decompositions $x = \sum a_i x_i$, $y = \sum b_i y_i$.  Moreover, it is independent of the choice of generic pearl data triple $\mathcal{D}(f, \rho, J)$.
\end{cor}

\section{Quantum torsion for $E^{1, \varphi}$-narrow Lagrangians}\label{sec:e1narrow}
Motivated by Buhovsky's Theorem \ref{thm:buhovsky}, we introduce the following class of pairs (Lagrangian, representation of $H_2(M, L)$), which contains $\varphi$-narrow tori:
\begin{dfn}\label{def:e1narrow}
Let $(L, \varphi)$ be a pair where $L$ is a closed, orientable and spin monotone Lagrangian (and assume these choices are fixed), and $\varphi \in \mathcal{N}(L, \mathbb{F})$ is a narrow representation.  Then $(L, \varphi)$ is called $E^{1, \varphi}$-narrow if the associated degree spectral sequence satisfies $H(E^{1, \varphi}_{*,*}, d^{1, \varphi}) \cong E_{*,*}^{2, \varphi} = 0$.
\end{dfn}
Recall that $d_{1*}^\varphi([L])=0$ for degree reasons.  Then, we see from the results in \S \ref{sec:d1complex} that $E_{*,*}^{2, \varphi}=0$ if and only if any of the following equivalent conditions hold:
\begin{itemize}
\item The unit $[L] \in H_n(L; \mathbb{F})$ for the Morse intersection product is in the image of $d_{1*}^\varphi$.
\item The family of chain complexes given by $(H_*(L;\F), d_{1*}^\varphi)$ (see (\ref{eq:d1complex})) is acyclic.
\end{itemize}
Assume that $\text{char } \F \nmid |\tor(L)|$.  For an $E^{1, \varphi}$-narrow pair $(L, \varphi)$ as above, each choice of pearl data $\mathcal{D} = (f, \rho, J)$ yields a family of equivalent bases $h^f_* \otimes_\F 1$ of $H^f_*(L;\mathbb{F})$, given by tensoring any basis of the free part of $H^f_*(L; \Z)$ (which are all equivalent for a fixed $f$) with $\F$, as explained in \S \ref{sec:milnorsdef}.  Since the complex (\ref{eq:d1complex}) is acyclic by assumption, its torsion relative to this equivalence class of basis is defined and we denote it by $\tau(E^{1, \varphi}, \mathcal{D}, h^f_* \otimes 1)$.  The choice of pearl data triple is also irrelevant:
\begin{lem}
Assume $(L, \varphi)$ is $E^{1, \varphi}$-narrow and suppose that $\text{char } \F \nmid |\tor(L)|$.  Given $\mathcal{D}_1, \mathcal{D}_2$ two generic pearl data triples, we have $\tau(E^{1, \varphi}, \mathcal{D}_1, h^{f_1}_*\otimes 1) = \tau(E^{1, \varphi}, \mathcal{D}_2, h^{f_2}_*\otimes 1)$.
\end{lem}
\begin{proof}
First observe that in the commutative diagram (\ref{eq:d1commuteMorse}), the Morse comparison morphisms over $\Z$ and $\F$ are related by the formula $[\Phi^{f_1}_{f_2}]_\F = [\Phi^{f_1}_{f_2}]_\Z \otimes_\F 1$, by assumption on the characteristic of $\F$.  In other words, a basis of the form $h^{f_1}_* \otimes 1$ is mapped by the Morse comparison morphism to a basis $h^{f_2}_* \otimes 1$.  Finally, writing the comparison morphisms in those bases, we get matrices such that $\det ([\Phi^{f_1}_{f_2}]_\Z \otimes_\F 1) = \det [\Phi^{f_1}_{f_2}]_\Z = \pm 1$, hence changing the Morse function does not change the value of torsion. 
\end{proof}
Therefore, it makes sense to speak of the torsion of the acyclic complex (\ref{eq:d1complex}) with respect to the equivalence class of $h_* \otimes 1$, which we denote by $\tau(E^{1, \varphi}, h_* \otimes 1) \in \K(\F)$. The following theorem gives a useful tool to compute quantum Reidemeister torsion, as well as to prove its invariance, and is the reason we introduced the class of $E^{1, \varphi}$-narrow pairs $(L, \varphi)$.  Compare with Abouzaid-Kragh \cite[Lemma 2.1]{Abou-Kragh:simplehomtypenearby}. Recall from \S \ref{sec:pearltorsion} that $\tau(L, \varphi)$ induces a function $$\tau(L, \cdot)\colon \{ \varphi \in \mathcal{N}^{\text{free}}(L, \F) \} \to \K(\F),$$ which depends on variables $z_1, \dots, z_{b_2(M, L)}$.  Moreover, the representations $\varphi$ such that $(L, \varphi)$ is $E^{1, \varphi}$-narrow form an open subset of $\mathcal{N}^{\text{free}}$, since their complement is given by the vanishing of the map $d_{1*}^\varphi \colon H_{n+1-N_L}(L;\F) \to H_n(L;\F) \cong \F$.
\begin{thm}\label{thm:main}
Let $L$ be a closed, monotone, orientable, spin Lagrangian and $\varphi \in \mathcal{N}(L, \mathbb{F})$ a narrow representation, where $\text{char } \F \nmid |\tor(L)|$.  Suppose that $(L, \varphi)$ is an $E^{1, \varphi}$-narrow pair.  Then the quantum Reidemeister torsion of $(L, \varphi)$ is independent of the choice of generic pearl data $\mathcal{D} = (f, \rho, J)$ and satisfies
$$\tau(L, \varphi) = \tau((L, \varphi), \mathcal{D}) = \dfrac{|\tor_{\text{ev}}(L)|}{|\tor_{\text{odd}}(L)|}\tau(E^{1, \varphi}, h_* \otimes 1) = \tau(E^{0, \varphi}, \text{Crit}_* f, h_*\otimes 1)\tau(E^{1, \varphi}, h_*\otimes 1)$$
where $\tau(E^{0, \varphi}, \text{Crit}_* f, h_*\otimes 1)$ is Milnor's torsion of $E^{0, \varphi}$.
Moreover, the function $\tau(L, \cdot)$ is a rational function whose coefficients are expressed in terms of genus zero open Gromov-Witten invariants of $L$.
\end{thm}
\begin{proof}
Assume first that $N_L=2$.  We deal with the general case at the end of Step 1 below.  We also fix $\mathcal{D}=(f, \rho, J)$ a generic pearl data triple and $\varphi \in \mathcal{N}(L, \mathbb{F})$ a narrow representation.  

Recall from \S \ref{sec:pearltorsion} that we compute Reidemeister torsion $\tau((L, \varphi), \mathcal{D})$ with respect to the equivalence class of bases $\{ c_i\}$ given by $\text{Crit}_{*}(f)$.  We will show that one can change this basis to a more manageable one without affecting the torsion.

\newcounter{Stepcounter}
\stepcounter{Stepcounter}
\noindent\textit{Step \arabic{Stepcounter}: \stepcounter{Stepcounter}}\\
Recall from (\ref{eq:dsum}) that the differential decomposes as a sum $d^\varphi =: d = d_M^\varphi + d_1^\varphi + \dots$, where $d_i^\varphi \colon \Ccal_k \to \Ccal_{k-1+2i}$, since we assume $N_L=2$.  Moreover, $d_M^\varphi$ is simply the reduction in $\F$ of the count of negative gradient flow lines between two critical points of consecutive index, induced by $\varphi(1_\Z)=1_\F$, hence it computes Morse homology $H_*^f(L;\F)$.  In matrix form, $d\colon \Ccal_{[i]} \to \Ccal_{[i-1]}$ is then:
\begin{equation}\label{eq:matrixd}
\begin{blockarray}{cccccc}
& \cdots & \Ccal_{i-2} & \Ccal_i & \Ccal_{i+2} & \cdots\\
\begin{block}{c(ccccc)}
\vdots & \cdots & \cdots & \cdots & \cdots & \cdots\\
\Ccal_{i-1} & \cdots & d_1^\varphi & d_M^\varphi & 0 & \cdots\\
\Ccal_{i+1} & \cdots & d_2^\varphi & d_1^\varphi & d_M^\varphi & \cdots\\
\Ccal_{i+3} & \cdots & d_3^\varphi & d_2^\varphi & d_1^\varphi & \cdots\\
\end{block}
\end{blockarray}
\end{equation}\\
In case $N_L \geq 4$, define $\bar{\Ccal}_{{2k}} = \Ccal_{kN_L} \oplus \Ccal_{kN_L+2}\oplus \cdots \oplus \Ccal_{(k+1)N_L-2}$ and  $\bar{\Ccal}_{2k+1} = \Ccal_{kN_L+1} \oplus \Ccal_{kN_L+3}\oplus \cdots \oplus \Ccal_{(k+1)N_L-1}$, for $k=0, 1, \dots, n/N_L$.  Then matrix (\ref{eq:matrixd}) will look exactly the same as above, with $\Ccal_i$ replaced by $\bar{\Ccal}_{i}$ and the rest of the proof will also be the same.  To lighten the notation, we assume that $N_L=2$ in the remaining steps below.

\noindent\textit{Step \arabic{Stepcounter}: Change the basis of $\Ccal_i$ to a more manageable one.\stepcounter{Stepcounter}}\\ 
Applying the constructions from \S \ref{sec:milnorsdef} to the Morse differential $d_M^\varphi$, we get two exact sequences of $\F$-vector spaces:
\begin{align}
&\xymatrix{0 \ar[r] & Z_i^M \ar[r] & \Ccal_i \ar[r]^d & B^M_{i-1} \ar[r] & 0}\\
&\xymatrix{0 \ar[r] & B^M_i \ar[r] & Z_i^M \ar[r] & H_i \ar[r] & 0} \nonumber
\end{align}
The superscript $M$ is used to emphasize that these are the Morse differential boundaries, and not the pearl boundaries.  Pick any basis $b_i^M$ of $B_i^M$ as well as the bases $h_i = h^f_i \otimes_\F 1$ of $H_i(L;\F)$.  We obtain a new basis of $\Ccal_i$ given by $s(h_i) b_i^M s(b_{i-1}^M)$, written sometimes $h_i b_i^M b_{i-1}^M$, where $s$ denotes sections of the sequences above, which by a slight abuse of notation are given the same name.  Set $b_{[i]}^M = \oplus_{k \equiv i \mod 2} b_{k}^M$ and $h_{[i]} = \oplus_{k \equiv i \mod 2} h_{k}, \; i=0,1$.  By formula (\ref{eq:tauchangebase}), we have 
$$\tau(\Ccal_{[*]}, b_{[*]}^M b_{[*-1]}^M h_{[*]}) = \tau((L, \varphi), \mathcal{D}) \prod_{i=0}^n [c_i/h_i b_i^M b_{i-1}^M]^{(-1)^i}$$
By formula (\ref{eq:torsioneqtorsion}), this yields
$$\tau(\Ccal_{[*]}, b_{[*]}^M b_{[*-1]}^M h_{[*]}) = \dfrac{|\tor_{\text{odd}}(L)|}{|\tor_{\text{ev}}(L)|}\tau((L, \varphi), \mathcal{D})$$
Our task is now to compute the torsion on the left-hand-side of this equality.

Note that $E^{0, \varphi}$ is the Morse complex, therefore $$\tau(E^{0, \varphi}, \text{Crit}_* f, h_*\otimes 1) = \dfrac{|\tor_{\text{ev}}(L)|}{|\tor_{\text{odd}}(L)|}$$

\noindent\textit{Step \arabic{Stepcounter}\stepcounter{Stepcounter}: Write $d_M^\varphi$ in the new basis.}\\
With respect to the bases $h_i b_i^M b_{i-1}^M$, the Morse differential $d_M^\varphi\colon \Ccal_i \to \Ccal_{i-1}$ is given by:
\begin{equation}\label{eq:matrixdM}
\begin{blockarray}{cccc}
& s(H_i) & B_i^M & s(B^M_{i-1})\\
\begin{block}{c(ccc)}
s(H_{i-1}) & 0 & 0 & 0\\
B_{i-1}^M    &0 & 0 & I\\
s(B^M_{i-2}) & 0 & 0 & 0\\
\end{block}
\end{blockarray}
\end{equation}
where $I$ is the identity matrix.\\

\noindent\textit{Step \arabic{Stepcounter}\stepcounter{Stepcounter}:  Write the matrix for $d_{1}^\varphi\colon \Ccal_i \to \Ccal_{i+1}$ in the new basis.}\\
As noted in \S \ref{sec:d1complex}, $d_1^\varphi$ induces a map $d_{1*}^\varphi$ on Morse homology.  The map $d_1^\varphi \colon \Ccal_i \to \Ccal_{i+1}$ is then, in matrix form:
\begin{equation}\label{eq:matrixd1}
\begin{blockarray}{cccc}
& s(H_i) & B_i^M & s(B^M_{i-1})\\
\begin{block}{c(ccc)}
s(H_{i+1}) & d_{1 *}^\varphi & 0 & ?\\
B_{i+1}^M    & ? & ? & ?\\
s(B^M_{i}) & 0 & 0 & ?\\
\end{block}
\end{blockarray}
\end{equation}
where $?$ denote some matrices which do not matter for our purpose.\\

\noindent\textit{Step \arabic{Stepcounter}\stepcounter{Stepcounter}:  Find a basis for the pearl boundaries $B_{[i]}=\text{Im}(d\colon \Ccal_{[i+1]}\to \Ccal_{[i]})$.}\\
Using matrices (\ref{eq:matrixdM}) and (\ref{eq:matrixd1}), we see that a part of matrix (\ref{eq:matrixd}) looks as follows:
\begin{equation}\label{eq:bigmatrixd}
\begin{blockarray}{c ccc  ccc  ccc}
 &  & \Ccal_{i-2} &  & & \Ccal_i & & & \Ccal_{i+2}\\
\begin{block}{c (ccc | ccc | ccc)}
            & d_{1*}^\varphi & 0 & ? & 0 & 0 & 0 & & &\\
\Ccal_{i-1}     & ? & ? & ? & 0 & 0 & I & & 0 & \\
            & 0 & 0 & ? & 0 & 0 & 0 & \\
\cline{2-10}
            &  &  &  & d_{1*}^\varphi & 0 & ? & 0 & 0 & 0\\
\Ccal_{i+1}     &  & d_2^\varphi & & ? & ? & ? & 0 & 0 & I\\
            &  &  &  & 0 & 0 & ? & 0 & 0 & 0\\
\end{block}
\end{blockarray}
\end{equation}
Pick a set of linearly independent vectors $j_r \subset s(H_{r}(L;\F))$ that are mapped by $d_{1*}^\varphi$ to a basis of $d_{1*}^\varphi(s(H_r)) \subset s(H_{r+1})$.  Abusing notation a bit, we say that $j_r$ is a basis for the image of $d_{1*}^{ \varphi}$.  We see, by the shape of matrix (\ref{eq:bigmatrixd}), that
$d\colon \Ccal_{[i]}\to \Ccal_{[i-1]}$ is injective on the vector space spanned by the linearly independent vectors in the set $V_{[i]}=\{ s(b_{r-1}^M), j_r \; | \; r \equiv [i] \mod 2 \}$.  In other words, $d(V_{[i]})$ is a set of linearly independent vectors in $\Ccal_{[i-1]}$ and $\dim B_{[i-1]} \geq \dim V_{[i]}$.  Since $L$ is $E^{1, \varphi}$-narrow, we have 
\begin{equation}\label{eq:imd1}
\ker d_{1*}^{r, \varphi} = \text{Im}(d_{1*}^{\varphi}\colon H_{r-1}(L;\mathbb{F}) \to H_{r}(L;\mathbb{F}))= \text{Im } d_{1*}^{r-1, \varphi}
\end{equation}
\begin{lem}\label{lem:basisd}
The vectors in $d(V_{[i]})$ span all of $B_{[i-1]}$, hence provide a basis for the image of $d$.
\end{lem}
\begin{proof}
Since $L$ is $\varphi$-narrow, we have $\dim \Ccal_{[i]} = \dim B_{[i]} + \dim B_{[i-1]}$.  Moreover, $L$ is $E^{1, \varphi}$-narrow, hence $\dim H_r(L) - \dim \ker d_{1*}^{r, \varphi} = \dim \text{Im } d_{1*}^{r, \varphi} = \dim \ker d_{1*}^{r+1, \varphi}$.  A dimension count then yields
\begin{align*}
\dim B_{[i]} + \dim B_{[i-1]} & \geq \dim V_{[i-1]} + \dim V_{[i]}\\
& = \sum_{r \equiv [i-1]} (\dim s(B_{r-1}^M) + \dim H_r - \dim \ker d_{1*}^{r, \varphi}) \\
& + \sum_{s \equiv [i]} (\dim s(B_{s-1}^M) + \dim H_s - \dim \ker d_{1*}^{s, \varphi})\\
&= \sum_{r \equiv [i-1]} (\dim s(B_{r-1}^M) + \dim \ker d_{1*}^{r+1, \varphi})\\
&+ \sum_{s \equiv [i]} (\dim s(B_{s-1}^M) + \dim H_s - \dim \ker d_{1*}^{s, \varphi})\\
&= \sum_{s \equiv [i]} \dim H_s + \dim s(B_{s-1}^M) + \dim s(B_{s}^M) \\
&= \sum_{s \equiv [i]} \dim \Ccal_{s} = \dim \Ccal_{[i]}.
\end{align*}
Therefore, the first inequality is an equality and, since $\dim B_{[i-1]} \geq \dim V_{[i]}$ for each $[i]$, we get the desired conclusion.
\end{proof}

\noindent\textit{Step \arabic{Stepcounter}: The torsion of $(E^{1, \varphi}, h_*\otimes 1)$.\stepcounter{Stepcounter}}\\ 
By assumption on $L$, we have the following exact sequences:
\begin{equation}\label{eq:d1starsplit}
\xymatrix{
0 \ar[r] & \text{Im } d^{r-1, \varphi}_{1*} \ar[r] & H_r(L;\mathbb{F}) \ar[r] & \text{Im } d^{r, \varphi}_{1*} \ar[r] & 0
}
\end{equation}
which split by a splitting $s\colon \text{Im } d^{r, \varphi}_{1*} \to H_r$ that we fix.  Denote $q_i = d_{1*}^{i-1, \varphi}(j_{i-1})$ the basis (i.e.\ a list of vectors) of $\text{Im } d_{1*}^{i-1, \varphi}$.  By Lemma \ref{lem:basisd}, a basis $b_{[i-1]}$ of $B_{[i-1]} = \text{Im } (d\colon \Ccal_{[i]} \to \Ccal_{[i-1]})$ is given by concatenating the vectors in the list $q_{[i-1]}$ and the vectors $d(s(b^M_{[i]}))$.  Therefore, using the splitting $s$ of the sequence (\ref{eq:d1starsplit}), the map $s_d\colon B_{[i-1]} \to \Ccal_{[i]}$, defined via
\begin{align*}
q_{[i-1]} & \to s(q_{[i-1]})\\
d(s(b^M_{[i]})) & \to s(b^M_{[i]})
\end{align*}
is a section of $d$.  Reordering the vectors in the basis $b_{[i]}$ in the order indicated by the labels above matrix (\ref{eq:bigmatrixdet}) (which does not change torsion, since it is an element of $\mathbb{F}^\times / \pm 1$), we see that the matrix $b_{[0]} b_{[1]} / h_{[0]}b^M_{[0]}b^M_{[1]}$ is upper triangular and the relevant diagonal parts of it look as follows:
\begin{equation}\label{eq:bigmatrixdet}
\begin{blockarray}{c c cccc c}
& \cdots & d_{1*}^\varphi(j_{2i-1}) & s_d(q_{2i+1}) & d(s(b_{2i}^M)) & s_d (d(s(b^M_{2i-1}))) & \cdots\\
\begin{block}{c(c | cccc | c)}
\underset{l \leq 2(i-1), \; [l] \equiv 0 \mod 2}{\bigoplus} \Ccal_{l} & \cdots & 0 & 0 & 0 & 0 & 0\\
\cline{2-7}
H_{2i}(L;\mathbb{F}) & \cdots & q_{2i} & s(q_{2i+1}) & 0 & 0 & 0\\
B_{2i}^M & \cdots & ? & 0 & I & 0 & 0\\
s(B_{2i-1}^M) & \cdots & 0 & 0 & 0 & I & 0\\
\cline{2-7}
\underset{l \geq 2(i+1), \; [l] \equiv 0 \mod 2}{\bigoplus} \Ccal_{l} & \cdots & ? & ? & ?& ? & \cdots\\
\end{block}
\end{blockarray}
\end{equation}
There is also a similar matrix for $b_{[0]} b_{[1]} / h_{[1]}b^M_{[1]}b^M_{[0]}$, with $q_{2i}$ (resp. $s(q_{2i+1})$) replaced with $q_{2i-1}$ (resp. $s(q_{2i})$).  Therefore, 
\begin{align*}
\tau((L, \varphi), \mathcal{D}) &= \dfrac{|\tor_{\text{ev}}(L)|[b_{[0]} b_{[1]} / h_{[0]}b^M_{[0]}b^M_{[1]}]}{|\tor_{\text{odd}}(L)|[b_{[0]} b_{[1]} / h_{[1]}b^M_{[1]}b^M_{[0]}]}= \dfrac{|\tor_{\text{ev}}(L)|}{|\tor_{\text{odd}}(L)|}\prod_{k=0}^n [q_k s(q_{k+1})/ h_k]^{(-1)^k}\\
&= \dfrac{|\tor_{\text{ev}}(L)|}{|\tor_{\text{odd}}(L)|}\tau(E^{1, \varphi}, h_*\otimes 1) = \tau(E^{0, \varphi}, \text{Crit}_* f, h_*\otimes 1)\tau(E^{1, \varphi}, h_*\otimes 1)
\end{align*}
and this proves the first part of the theorem.

\noindent\textit{Step \arabic{Stepcounter}: Open GW-invariants.\stepcounter{Stepcounter}}\\
The last step of the proof is now simply a matter of unwrapping definitions.

Fix a generic pearl triple $(f, \rho, J)$ and a basis $h_* \otimes 1$ associated to $f$. Then, by Proposition \ref{prop:d1openGW}, $d_{1*}^\varphi$ is represented by a matrix whose entries are
$$\sum_A GW_{0,2}^A(h_k \otimes 1, h_{k-1+N_L}\otimes 1)\varphi(A)$$

Finally, a basis for the image of $d_{1*}^\varphi$ is obtained by applying Gau\ss\ algorithm to the matrix $d_{1*}^\varphi$.  In this algorithm, one performs only rational operations (over $\mathbb{F}$) on the entries of the matrix.  The same holds when taking a section of the sequence (\ref{eq:d1starsplit}).  Torsion is then obtained by taking products and divisions of determinants, which are themselves polynomial functions in the entries of the matrices involved.
\end{proof}

\begin{rem}
It seems that the theorem should generalize to yield the following formula:
$$\tau((L, \varphi), \mathcal{D}) = \prod_{k=0}^n \tau(E^{k, \varphi})$$
which would prove invariance for all monotone Lagrangians.  However it is not clear how to choose bases for the pages $E^{k, \varphi}, \; k \geq 2$.  In the theorem above, there was a canonical choice of bases given by $\Crit_* \; f$ and $h_* \otimes 1$ on $E^0$ and $E^1$.  See also \cite[Remark 4.2.3]{Cha:3fold} for a similar discussion.
\end{rem}
\begin{proof}[Proof of Corollary \ref{cor:thmfirst}]
This is similar to Biran-Cornea \cite[Proposition 4.3.1]{Bi-Co:rigidity}.  The only difference concerns orientations of the space of pearls, which ultimately relies on orientations of the space of pseudo-holomorphic discs. But our choice of spin structures on $L$ and $g(L)$ guarantees that orientations are also preserved; this follows e.g.\ from Cho \cite[Theorem 6.4]{Cho:Clifford} or Fukaya-Oh-Ohta-Ono \cite[\S 8.1.4]{FO3:book-vol2}. Here are the details.

Fix $\mathcal{D} = (f, \rho, J)$ a generic pearl data triple on $L$, as well as an orientation and a spin structure on $L$, and $g \in \text{Symp}(M, \omega)$.  Then $g_*(\mathcal{D}) = (f \circ g^{-1}, (g^{-1})^*(\rho), g_* J g_*^{-1})$ is a generic pearl data triple on $g(L)$.  Moreover, with the orientation and spin structure on $g(L)$ induced from the ones on $L$ by $g$, using the results of Cho and Fukaya et al.\ , we obtain a chain-level identification $g\colon \Ccal(L, \mathcal{D}) \to \Ccal(g(L), g_*(\mathcal{D}))$.  Therefore, torsion on both sides coincide, as well as the degree spectral sequences, open Gromov-Witten invariants, etc.\
\end{proof}

We end this section with a useful computational lemma, making more precise the coefficients of the rational function mentioned above.
\begin{lem}\label{lem:splittingofd1}
Let $(L, \varphi)$ be an $E^{1, \varphi}$-narrow pair as above and $[L]:=([L]_\Z \otimes 1) \in H_{n}(L;\mathbb{F}) \cong \mathbb{F}$ denote the fundamental class.  Fix $\sigma \in H_{n+1-N_L}(L;\mathbb{F})$ a preimage of $[L]$ with respect to $d_{1*}^\varphi$.  Then a splitting of the exact sequence (\ref{eq:d1starsplit}) is given by left-multiplication with $\sigma$:
\begin{align*}
s\colon (\text{Im } d_{1*}^\varphi \subset H_k(L;\mathbb{F})) & \to H_{k +1-N_L}(L; \mathbb{F})\\
x & \mapsto \sigma \cdot x
\end{align*}
where $\cdot$ denotes the Morse intersection product.
\begin{proof}
This is a simple application of the Leibniz rule in Theorem \ref{thm:specseq}.  Indeed, we have $x=d_{1*}^\varphi(y)$ for some $y$, by assumption.  Moreover, $d_{1*}^\varphi(\sigma \cdot x) = d_{1*}^\varphi(\sigma) \cdot x + (-1)^{N_L-1} \sigma \cdot d_{1*}^\varphi(x) = d_{1*}^\varphi(\sigma) \cdot x = [L] \cdot x = x$
\end{proof}
\end{lem}

\subsection{A special class of $E^{1, \varphi}$-narrow pairs and their torsion}\label{sec:assumptionstar}
In this section, we assume that the monotone Lagrangian $L$ is a product $S^{2k+1} \times V \subset (M, \omega)$.  Moreover, in order to apply Theorem \ref{thm:main}, we consider only connected closed orientable $V$ and we assume that $L$ admits a spin structure, which is equivalent in this case to $V$ admitting a spin structure; we also allow $V$ to be a point.  Let $\sigma:= *\otimes [V] \in H_{n-(2k+1)}(L;\Z)$ be the fundamental class of $V$ in $L$.  We impose the following rather strong restriction on open Gromov-Witten invariants of $L$:

\noindent \textbf{Assumption $(\star)$}: Let $\varphi\colon H_2(M, L) \to \mathbb{F}^\times$ be a representation such that $d_{1*}^\varphi(\sigma) = r_\varphi L, \; 0 \neq r_\varphi \in \mathbb{F}.$\\
Obviously, $L$ is then $E^{1, \varphi}$-narrow and $N_L=2k+2$.  Notice also that
$$r_\varphi = \sum_{\mu(A) = 2k+2} GW_{0,2}^A(\sigma, pt)\varphi(A)$$
by Proposition \ref{prop:d1openGW}.
\begin{thm}\label{thm:assumptionstar}
Let $L=S^{2k+1}\times V$ be a closed, orientable, spin and monotone Lagrangian submanifold.  Let $\varphi\colon H_2(M, L)\to \F^\times$ be a representation satisfying Assumption $(\star)$.  Suppose that $\text{char } \F \nmid |\tor(L)|$.  Then $N_L=2k+2$, $L$ is $E^{1, \varphi}$-narrow and $\tau(L, \varphi) = \dfrac{|\tor_{\text{ev}}(L)|}{|\tor_{\text{odd}}(L)|}r_\varphi^{-\chi(V)} \in \K(\F)$.
\end{thm}
In \S \ref{sec:ex}, we will provide non trivial applications of the above theorem and prove the following
\begin{cor}\label{cor:assumptionstar}
Assumption $(\star)$ is satisfied for every narrow representation $\varphi$ whenever $k=0$ and
\begin{itemize}
\item $V$ is a point.
\item $V$ is a $m$-torus, i.e.\ $L$ is a $m+1$-torus. Then $\tau(L, \varphi) = 1$
\item $V$ is an orientable surface of genus $g \geq 1$.  Then $\tau(L, \varphi) = r_\varphi^{2(g-1)}$
\item $L$ is given by a product Lagrangian embedding $S^1 \times V \subset (S^2 \times X, \omega_{S^2}\oplus \omega_X)$, where $H_2(X, V) = 0$.
\end{itemize}
\end{cor}
\begin{remsnonum}
a.  The case of $S^1\times S^2$ does not fit the above pattern and will be discussed in \S \ref{sec:s1s2}.\\
b.  The second and third cases follow from Buhovski's Theorem \ref{thm:buhovsky}.  We will nevertheless (re)prove them in \S \ref{sec:proofcorassumptionstar}.
\end{remsnonum}
\begin{proof}[Proof of Theorem \ref{thm:assumptionstar}]
Let $L=S^{2k+1} \times V$ and assume there exists $\varphi$ satisfying Assumption $(\star)$.  Note first that, by the K\"unneth formula, we have $H_r(L;\F) = (H_0(S^{2k+1})\otimes H_r(V)) \bigoplus (H_{2k+1}(S^{2k+1})\otimes H_{r-(2k+1)}(V))$

Let $\beta_{r, i}, \; i=1, \dots, \dim H_{r}(V)$ denote the basis of $H_r(V;\F)$ obtained by tensoring a basis of the free part of $H_r(V;\Z)$ with $1 \in \F$.  Recall that all bases obtained in this way belong to the same equivalence class.  Similarly, let $*$ and $[S^{2k+1}]$ denote the corresponding bases of $H(S^{2k+1};\F)$.  Then, a basis of $\text{Im } (d_{1*}^\varphi \colon H_{r-(2k+1)}(L) \to H_{r}(L)), \; r \leq n,$ is given by the set of vectors
$$q_{r, i}:= \frac{1}{r_\varphi}d_{1*}^\varphi(* \otimes \beta_{r-(2k+1), i})$$
Indeed, by Assumption $(\star)$, we have
\begin{align*}
q_{r, i} & = \frac{1}{r_\varphi}d_{1*}^\varphi(\sigma \cdot ([S^{2k+1}]\otimes \beta_{r-(2k+1), i}))\\
& = \frac{1}{r_\varphi}d_{1*}^\varphi(\sigma) \cdot ([S^{2k+1}]\otimes \beta_{r-(2k+1), i}) - \frac{1}{r_\varphi}\sigma \cdot d_{1*}^\varphi([S^{2k+1}]\otimes \beta_{r-(2k+1), i})\\
& = [S^{2k+1}]\otimes \beta_{r-(2k+1), i} - \frac{1}{r_\varphi}\sigma \cdot d_{1*}^\varphi([S^{2k+1}]\otimes \beta_{r-(2k+1), i})
\end{align*}
Also, notice that by definition of $\sigma = *\otimes [V]$, we have, for any $x\in H_*(L)$, $\sigma \cdot x \in H_{0}(S^{2k+1})\otimes H_{*}(V) \subset H_{*}(L)$, which proves linear independence of the vectors $q_{r, i}$.  To prove that the vectors span the image, note that the above computation yields
$$\frac{1}{r_\varphi} \sigma \cdot q_{r+2k+1, j} = \frac{1}{r_\varphi}* \otimes \beta_{r, j}$$
since $\sigma^2 = 0$.  Therefore, the set of vectors $q_{r, i}$ and $\frac{1}{r_\varphi}\sigma \cdot q_{r+2k+1, j}$ provide a basis of $H_{r}(L)$, where $i=1, \dots, \dim H_{r}(V)$, $j=1, \dots, \dim H_{r+2k+1}(V)$.  Hence the vectors $q_{r, i}$ must span the image.

Finally, we have the change of basis matrix $q_r s(q_{r+2k+1}) / (*\otimes \beta_r + [S^{2k+1}]\otimes \beta_{r-(2k+1)})$ given by $$\bordermatrix{ & \{q_{r, i} \} & \{ s(q_{r+2k+1, j}) \} \cr
                H_0(S^{2k+1})\otimes H_{r}(V) & ? &  \frac{1}{r_\varphi} I \cr
                H_{2k+1}(S^{2k+1})\otimes H_{r-(2k+1)}(V) & I &  0 \cr
}$$
with determinant $\dfrac{1}{r_\varphi^{\dim H_r(V;\F)}}$.  By Theorem \ref{thm:main}, we conclude that
$$\tau(L, \varphi) =\dfrac{|\tor_{\text{ev}}(L)|}{|\tor_{\text{odd}}(L)|} r_\varphi^{-\chi(V;\F)}$$
\end{proof}

\section{Examples and proof of Corollary \ref{cor:assumptionstar}}\label{sec:ex}
\subsection{The circle}\label{sec:circle}
Although it is possible to compute this example using only a perfect Morse function (by Theorem \ref{thm:main}), we will do everything by brute force since it illustrates how computations are done in general.

Consider first a monotone circle $S^1 \subset S^2$, i.e.\ a circle dividing $S^2$ in two parts (or hemispheres) of equal area.  Then $H_2(S^2, S^1) \cong \Z^2$, with generators given by the classes of the hemispheres, denoted by $A$ and $B$.  Fix a Morse-Smale function $f_n\colon S^1 \to \R$ with $\text{Crit}_0 f_n = \{ y_i \; | \; i=1, \dots, n\}$ and $\text{Crit}_1 f_n = \{ x_i \; | \; i=1, \dots, n\}$, as on Figure \ref{fig:morses1}.  
\begin{figure}[h!]
  \centering
    \includegraphics[width=5cm]{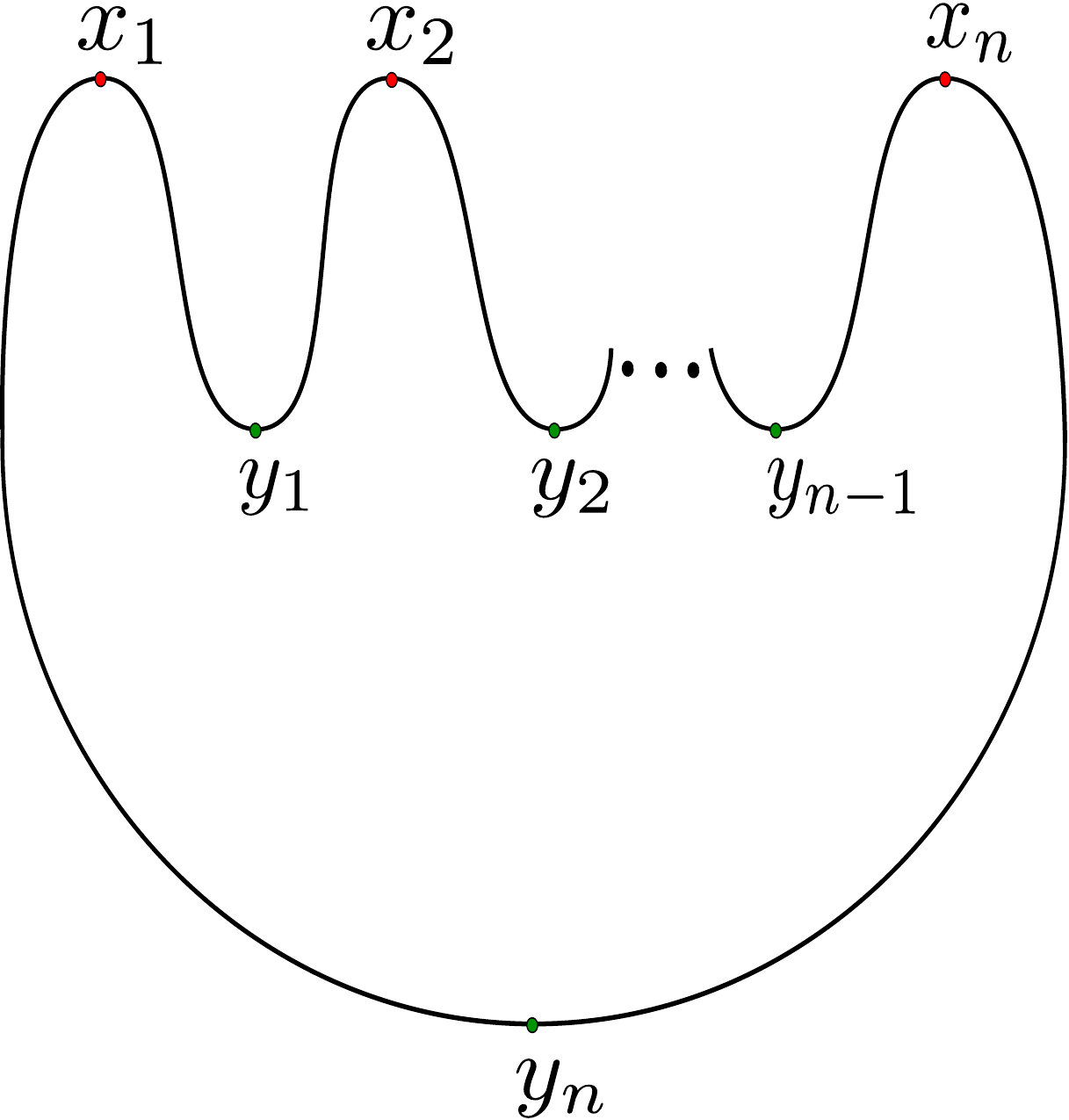}
      \caption{$f_n\colon S^1 \to \R$, the "rock-and-roll" function}
      \label{fig:morses1}
\end{figure}

By varying $n$, this gives all possible Morse-Smale functions on $S^1$.  Fix also a generic compatible almost complex structure $J \in \mathcal{J}_{\omega}$.

The pearl differential on $\Z[H_2(S^2, S^1)]\langle \text{Crit} f_n\rangle$ is then given by
$$d(x_i) = y_{i-1} - y_i, \quad d(y_i) = (A-B)\sum_{j=1}^n x_j$$
where we used the convention $y_0 = y_n$.  To understand the shape of $d(y_i)$, note that, given any two points on $S^1$ (e.g. $y_i$ and $x_j$) and any generic $J$, there is one simple pseudoholomorphic disc in the class $A$ and minus one in the class $B$, going through $x_j$ and $y_i$, counted with appropriate signs.

Given $\varphi\colon H_2(S^2, S^1) \to \mathbb{F}^{\times}, \; A \mapsto z_1, \; B \mapsto z_2$, we get an induced representation $\varphi\colon \Z[H_2(S^2, S^1)]\to \mathbb{F}$ and $S^1$ is $\varphi$-narrow if and only if $r=z_1 -z_2\neq 0$, since $d(y_i)=r\sum x_j$.

A quick check shows that bases $b_0$ and $b_1$ for $d(\text{Crit}_1 f)$ and $d(\text{Crit}_0 f)$ are given by
$$b_0 = \{ y_i - y_{i+1} \; | \; i=1, \dots, n-1\}, \quad b_1= \sum_{i=1}^n x_i.$$

Moreover, splittings of (\ref{eq:split}) can be chosen as
\begin{align*}
s\colon b_0  &\to C_1  & s\colon b_1 &\to C_0\\
y_i - y_{i+1}  &\mapsto x_{i+1} & \sum x_i &\mapsto \frac{y_1}{r}
\end{align*}
so that
$$
\begin{blockarray}{ccccccc}
                    & s(\sum x_i)     & (y_1-y_2)    & (y_2-y_3)& \cdots    & (y_{n-2} - y_{n-1})& (y_{n-1}-y_n)\\ 
 \begin{block}{c(c|ccccc)}
   y_1              & \frac{1}{r}  &  1           &   0       &  \cdots &0 & 0 \\ 
   \cline{2-7} 
   y_2              &    0        &   -1     &   1             &  \\  
   y_3              &    0        &          &   -1            & 1 \\
   \vdots           &    \vdots   &  &                & \ddots & \ddots\\
   y_{n-1}          & & & & & -1 & 1\\
   y_n              &    0        &        &          &      &  & -1\\
  \end{block}
\end{blockarray}= s(b_1) b_0 / c_0
$$
$$
\begin{blockarray}{ccccc}
                    & s(y_1-y_2)  & \cdots & s(y_{n-1}-y_n) & \sum x_i\\ 
 \begin{block}{c(ccc|c)}
   x_1              & 0           & \cdots &   0            &  1\\ 
   \cline{2-5} 
   x_2              &             &        &                &  1\\  
   \vdots           &             &\text{\Huge{I}}&                &  \vdots\\
   x_n              &             &        &                &  1\\
  \end{block}
\end{blockarray}= s(b_0) b_1 / c_1
$$

Finally, we arrive at 
$$\tau(S^1, \varphi) = \dfrac{1}{r} = \dfrac{1}{z_1 - z_2} \in \K(\F).$$  Note that this holds for any generic triple defining the pearl complex, as any Morse function on $S^1$ is like $f_n$ for some $n$, hence $\tau$ does not depend on the choice of generic data, a fact we knew already.

If $\varphi$ is induced from a representation  $\psi\colon H_1(S^1) \to \mathbb{F}^{\times}, \; 1 \mapsto z\neq 1,$ via $\varphi = \partial \circ \psi$, where $\partial\colon H_2(S^2, S^1) \to H_1(S^1)$ is the connecting morphism, then 
$$\tau(S^1, \varphi)= \dfrac{1}{z-\frac{1}{z}} = \dfrac{z}{z^2 -1} = \dfrac{z}{(z-1)(z+1)}$$
In analogy with the classical notion of Reidemeister torsion (or R-torsion), where one quotients $\K(\F)$ by $\varphi(H_2(S^2, S^1))$, we have:
$$\tau(S^1, \varphi) = \dfrac{1}{z^2-1} \in \K(\F) / \pm \varphi(H_2(S^2, S^1)) = \K(\F) / \pm \psi(H_1(S^1)).$$
This does not coincide with the usual value of R-torsion for the circle: $\Delta_{\psi}(S^1) = \dfrac{1}{z-1} \in \mathbb{F}^{\times}/ \pm \psi(H_1(S^1))$ (see e.g. \cite[\S 8]{Mil:torsion}).  Note however that $\Delta_{\psi}(S^1)$, as a function of $z$, divides $\tau(S^1, \varphi)$ (compare with \cite[Theorem 1.12]{HutLee:torsion} or \cite[Corollary 2.3.4]{Lee:torsion1}).

\begin{remnonum}
If $\Sigma_g$ is an orientable surface of genus different than zero, closed or not, then any simple closed contractible curve is a monotone Lagrangian $L$.  Moreover, it bounds only one holomorphic disc of Maslov index two (by assumption on the genus) and, using arguments similar to the ones above, we get:
$\tau(S^1, \varphi) = 1 \in \mathbb{F}^{\times}/ \varphi(H_2(\Sigma_g, S^1))$ and $\Delta_{\psi}(S^1)$, as a function of $z$, does not divide $1$!
\end{remnonum}

\subsection{Narrow $S^1 \times S^2$}\label{sec:s1s2}
In this section, we assume that $L=S^1\times S^2$ is monotone, and we fix pearl data $\mathcal{D}$ given by a perfect Morse function on $L$.  Fix also a narrow representation $\varphi\colon H_2(M, S^1\times S^2) \to \F^\times$.  In this case the pearl complex is $H_*(L) = \mathbb{F}\langle ab, a, b, L \rangle$, where $\deg a=1, \deg b=2$ and $ab$ denotes the intersection product.  Notice that Theorem \ref{thm:buhovsky} does not apply to $L$.  Nevertheless, it turns out that $L$ must be $E^{1, \varphi}$-narrow, see \cite{Cha:3fold}.

This case is covered by Theorem \ref{thm:assumptionstar}.   We have $d_{1*}(b) = d(b) = r_\varphi L$ and $\tau_\varphi(L) = r_\varphi^{-2} \in \K(\F)$, $r_\varphi \neq 0$.  An explicit example where this happens is given in Oakley-Usher \cite[Proposition 8.2]{Oa-Ush:certainlags}, based on Biran's Lagrangian circle bundle construction \cite{Bi:Nonintersections}.  In this case, $L\subset \C P^3$  and there is a single holomorphic disc contributing to $r_\varphi$, therefore
$$\tau(S^1\times S^2, \varphi)\equiv 1 \in \mathbb{F}^\times / \pm\varphi(H_2(M, L))$$
In the light of Fukaya's symplectic s-cobordism conjecture in \S \ref{sec:fukayascob}, it is tempting to conjecture that this Lagrangian is displaceable by a Hamiltonian isotopy.

\subsection{Proof of Corollary \ref{cor:assumptionstar}}\label{sec:proofcorassumptionstar}
\subsubsection{The $n$-torus, $n \geq 2$}\label{sec:torus}
By Theorem \ref{thm:buhovsky}, monotone tori are $E^{1, \varphi}$-narrow for every narrow representation $\varphi$ and their minimal Maslov number is two.  In this section, we consider a monotone Lagrangian $n$-torus $L = S^1 \times \cdots \times S^1$ ($n$ times), $n \geq 2$ and we fix $\varphi \in \mathcal{N}(L, \mathbb{F})$ a narrow representation.

By Theorem \ref{thm:main}, it is enough to compute the torsion of the exact sequence
$$\xymatrix{0 \ar[r] & H_0(L) \ar[r]^{d_{1*}^\varphi} & H_{1}(L) \ar[r] & \cdots \ar[r] & H_{n-1}(L) \ar[r]^{d_{1*}^\varphi} & H_n(L) \ar[r] & 0}$$
The Morse homology ring of $L$ is generated by $n$ classes $x_i \in H_{n-1}(L;\mathbb{F})$.  The unit is denoted by $L \in H_{n}(L)$.  Since the above sequence is exact, we have that $d_{1*}^\varphi(x_i) = r_i L$, $r_i \in \mathbb{F}$, and at least one of the $r_i$ is not zero.  Moreover, a permutation of the $x_i$ will not change the value of torsion as an element of $\K(\F)$, hence we assume without loss of generality that $r_1 \neq 0$.   Note also that $x_1$ is the fundamental class of an $n-1$-torus $V$ embedded in $L$ such that $L=S^1\times V$, therefore $L$ satisfies Assumption $(\star)$.  This proves the second part of Corollary \ref{cor:assumptionstar}.

Since $\tau(L, \varphi)$, a second-order invariant, is always trivial for $\varphi$-narrow monotone tori, it would be interesting to know if higher order invariants can be defined and whether they are trivial or not for such tori.  Compare with \cite{Fu:scobordism} and the discussion in \S \ref{sec:fukayascob}, which conjectures that there should be no higher-order invariant.

\subsubsection{The product of a circle and an orientable genus $g$ surface, $g \geq 2$}\label{sec:3manifolds}
We consider monotone embeddings of $L=S^1 \times \Sigma_g$, where $\Sigma_g$ is a closed orientable surface of genus $g$, $g \geq 2$ (the case $g=1$ was treated in the previous example).

Fix $\varphi \in \mathcal{N}(L, \mathbb{F})$ a narrow representation.  By Theorem \ref{thm:buhovsky}, $L$ is $E^{1, \varphi}$-narrow and $N_L=2$.  As a ring, we have
$$H_*(L) = \dfrac{\Lambda_\mathbb{F}[\alpha_1, \beta_1, \dots, \alpha_g, \beta_g, z]}{\alpha_1 \beta_1 = \cdots = \alpha_g \beta_g, \; \alpha_i \beta_j = 0 \; (i \neq j)}$$
where $\Lambda_{\mathbb{F}}[x, y]$ denotes the free graded exterior algebra (over $\mathbb{F}$) generated by $x,y$.  By exterior algebra, we mean here that $xy = (-1)^{(n-|x|)(n-|y|)} yx$, since we use the homology product.  Moreover, we have $\deg \alpha_i = \deg \beta_j = \deg z = 2$.

Set $d_{1*}^\varphi(\alpha_i) = a_i L, \; d_{1*}^\varphi(\beta_i) = b_i L, \; d_{1*}^\varphi(z) = c L$, where $a_i, b_j, c \in \mathbb{F}$.  Notice that, for $i\neq j$ (which is allowed since $g\geq 2$), we have $0 = d_{1*}^\varphi(\alpha_i \beta_j) = a_i \beta_j - b_j \alpha_i$, therefore $a_i = b_j = 0$ for all $i, j$.  Since $L$ is $\varphi$-narrow, the only possibility is that $c \neq 0$.  The following table determines $d_{1*}^\varphi$ completely:\\
\begin{center}
\begin{tabular}{|c|c|}
\hline
 $x \in H_k(L)$ & $d_{1*}^\varphi(x)$\\
\hline
$L$ & 0\\
\hline
$\alpha_i, \; \beta_i$ & 0\\
$z$ & $cL, \; c \neq 0$\\
\hline
$\alpha_1 \beta_1$ & 0\\
$\alpha_i z$ & $-c \alpha_i$\\
$\beta_i z$ & $-c \beta_i$\\
\hline
$\alpha_1 \beta_1 z$ & $c \alpha_1 \beta_1$\\
\hline
\end{tabular}
\end{center}
Notice that $z = * \times [\Sigma_g]$ and $c= \sum_A GW_{0,2}^A(c, pt)\varphi(A)$, therefore $\varphi$ satisfies Assumption $(\star)$ and, by Theorem \ref{thm:assumptionstar}, we have
$$\tau(L, \varphi) = c^{2(g-1)} = c^{-\chi(\Sigma_g)} \in \K(\F).$$
Of course, determining the exact value of $c$ depends on $M$ and on the Lagrangian embedding of $L$.

\subsubsection{Split Lagrangian embeddings $S^1 \times V \subset S^2\times X$, where $H_2(X, V) = 0$}\label{sec:splitlag}
We consider the monotone embedding of $L=S^1 \times V$ in $(M, \omega) = (S^2 \times X, \omega_{S^2}\oplus \omega_X)$ given by the product Lagrangian embeddings of $S^1$ in $S^2$ with a given Lagrangian embedding of $V \subset X$, where $V$ is closed and $H_2(X, V) = 0$.  We will use the notations from \S \ref{sec:circle}.  

Note that $H_2(M, L) \cong H_2(S^2, S^1)\oplus H_2(X, V)$ and the morphisms $\omega, \mu$ respect the splitting; moreover, $N_L=2$.  Choosing a (regular!) split compatible almost complex structure, we see that the only pseudoholomorphic discs with boundary on $L$ are pairs $(u, pt)$, where $u\colon(D^2, S^1)\to (S^2, S^1)$ is pseudoholomorphic and $pt$ is a constant disc in the pair $(X, V)$.  Therefore, 
$$d_{1*}(* \times [V]) = (A-B) L$$
where $(A-B) \in \Z[H_2(M, L)]$ and $A, B$ are the generators of $H_2(S^2, S^1)$.  Hence $L$ is $E^1$-narrow, since it is $\varphi$-narrow as soon as $S^1 \subset S^2$ is. As in \S \ref{sec:circle}, set $0 \neq r_\varphi = \varphi(A-B) \in \mathbb{F}^\times$.  We have $d_{1*}(* \times [V]) = r_\varphi L $ and $L$ satisfies Assumption $(\star)$.  Finally, using Theorem \ref{thm:assumptionstar}, we get
$$\tau(L, \varphi) = r_\varphi^{-\chi(V; \F)} \in \K(\F).$$
By using the computations from \S \ref{sec:circle}, we see that this torsion is not trivial as an element of either $\K(\F)$ or $\mathbb{F}^\times/ \varphi(H_2(M, L))$, provided that $\chi(V;\F) \neq 0$.  This completes the proof of Corollary \ref{cor:assumptionstar}.

\bibliographystyle{alpha}
\bibliography{../../../biblio_latex/bibliography}

\def\cprime{$'$} \def\cprime{$'$}
\begin{thebibliography}{FOOO09}

\bibitem[AK16]{Abou-Kragh:simplehomtypenearby}
M.~Abouzaid and T.~Kragh.
\newblock Simple homotopy equivalence of nearby {L}agrangians.
\newblock Preprint, can be found at \url{}https://arxiv.org/abs/1603.05431,
  2016.

\bibitem[BC09a]{Bi-Co:Yasha-fest}
P.~Biran and O.~Cornea.
\newblock A {L}agrangian quantum homology.
\newblock In {\em New perspectives and challenges in symplectic field theory},
  volume~49 of {\em CRM Proc. Lecture Notes}, pages 1--44. Amer. Math. Soc.,
  Providence, RI, 2009.

\bibitem[BC09b]{Bi-Co:rigidity}
P.~Biran and O.~Cornea.
\newblock Rigidity and uniruling for {L}agrangian submanifolds.
\newblock {\em Geom. Topol.}, 13(5):2881--2989, 2009.

\bibitem[BC12]{Bi-Co:lagtop}
P.~Biran and O.~Cornea.
\newblock Lagrangian topology and enumerative geometry.
\newblock {\em Geom. Topol.}, 16(2):963--1052, 2012.

\bibitem[Bir06]{Bi:Nonintersections}
P.~Biran.
\newblock Lagrangian non-intersections.
\newblock {\em Geom. Funct. Anal.}, 16(2):279--326, 2006.

\bibitem[BM15]{Bi-Me:cubic}
P.~Biran and C.~Membrez.
\newblock The {L}agrangian cubic equation.
\newblock Preprint. Available at \url{http://arxiv.org/abs/1406.6004v2}, 2015.

\bibitem[Buh09]{Bu:toripreprint}
L.~Buhovsky.
\newblock The {M}aslov class of {L}agrangian tori and quantum products in
  {F}loer cohomology.
\newblock Preprint, can be found at \url{https://arxiv.org/abs/math/0608063},
  2009.

\bibitem[Buh10]{Bu:toriaudin}
L.~Buhovsky.
\newblock The {M}aslov class of {L}agrangian tori and quantum products in
  {F}loer cohomology.
\newblock {\em J. Topol. Anal.}, 2(1):57--75, 2010.

\bibitem[Cha15]{Cha:3fold}
F.~Charette.
\newblock On the cohomology of narrow {L}agrangian 3-manifolds, quantum
  {R}eidemeister torsion, and the {L}andau-{G}inzburg superpotential.
\newblock Available at \url{http://arxiv.org/abs/1503.00460v2}, 2015.

\bibitem[Cho04]{Cho:Clifford}
C.-H Cho.
\newblock Holomorphic discs, spin structures, and {F}loer cohomology of the
  {C}lifford torus.
\newblock {\em Int. Math. Res. Not.}, 2004(35):1803--1843, 2004.

\bibitem[Coh73]{Coh:simple}
M.M. Cohen.
\newblock {\em A course in simple-homotopy theory}.
\newblock Springer-Verlag, New York-Berlin, 1973.
\newblock Graduate Texts in Mathematics, Vol. 10.

\bibitem[FOOO09]{FO3:book-vol2}
K.~Fukaya, Y.-G. Oh, H.~Ohta, and K.~Ono.
\newblock {\em Lagrangian intersection {F}loer theory: anomaly and obstruction.
  {P}art {II}}, volume~46 of {\em AMS/IP Studies in Advanced Mathematics}.
\newblock American Mathematical Society, Providence, RI, 2009.

\bibitem[Fra08]{Fr:msc}
U.~Frauenfelder.
\newblock Gromov convergence of pseudo-holomorphic disks.
\newblock {\em Journal of Fixed Point Theory and its Applications},
  (3):215--271, 2008.

\bibitem[Fuk97]{Fu:scobordism}
K.~Fukaya.
\newblock The symplectic {$s$}-cobordism conjecture: a summary.
\newblock In {\em Geometry and physics ({A}arhus, 1995)}, volume 184 of {\em
  Lecture Notes in Pure and Appl. Math.}, pages 209--219. Dekker, New York,
  1997.

\bibitem[GZ13]{Georg-Zinger:realGWtheorygenera}
P.~Georgieva and A.~Zinger.
\newblock Lagrangian circle actions.
\newblock Preprint. Available at \url{http://arxiv.org/abs/1307.8196}, 2013.

\bibitem[HL99]{HutLee:torsion}
M.~Hutchings and Y.-J. Lee.
\newblock Circle-valued {M}orse theory, {R}eidemeister torsion, and
  {S}eiberg-{W}itten invariants of {$3$}-manifolds.
\newblock {\em Topology}, 38(4):861--888, 1999.

\bibitem[Hut02]{Hut:torsion}
M.~Hutchings.
\newblock Reidemeister torsion in generalized {M}orse theory.
\newblock {\em Forum Math.}, 14(2):209--244, 2002.

\bibitem[Laz11]{Laz:decomp}
Laurent Lazzarini.
\newblock Relative frames on {$J$}-holomorphic curves.
\newblock {\em J. Fixed Point Theory Appl.}, 9(2):213--256, 2011.

\bibitem[Lee03]{Lee:lagtorsion}
Y.-J. Lee.
\newblock Non-contractible periodic orbits, {G}romov invariants, and
  {F}loer-theoretic torsions.
\newblock Preprint, can be found at \url{http://arxiv.org/abs/math/0308185},
  2003.

\bibitem[Lee05a]{Lee:torsion1}
Y.-J. Lee.
\newblock Reidemeister torsion in {F}loer-{N}ovikov theory and counting
  pseudo-holomorphic tori. {I}.
\newblock {\em J. Symplectic Geom.}, 3(2):221--311, 2005.

\bibitem[Lee05b]{Lee:torsion2}
Y.-J. Lee.
\newblock Reidemeister torsion in {F}loer-{N}ovikov theory and counting
  pseudo-holomorphic tori. {II}.
\newblock {\em J. Symplectic Geom.}, 3(3):385--480, 2005.

\bibitem[Liu04]{Liu:openGW}
C.-C.~M. Liu.
\newblock Moduli of {J}-holomorphic curves with {L}agrangian boundary
  conditions and open {G}romov-{W}itten invariants for an ${S}^1$-equivariant
  pair.
\newblock Preprint, can be found at \url{https://arxiv.org/abs/math/0210257},
  2004.

\bibitem[Mil66]{Mil:torsion}
J.~Milnor.
\newblock Whitehead torsion.
\newblock {\em Bull. Amer. Math. Soc.}, 72:358--426, 1966.

\bibitem[MS04]{McD-Sa:Jhol-2}
D.~McDuff and D.~Salamon.
\newblock {\em {$J$}-holomorphic curves and symplectic topology}, volume~52 of
  {\em American Mathematical Society Colloquium Publications}.
\newblock American Mathematical Society, Providence, RI, 2004.

\bibitem[Oh96]{Oh:spectral}
Y.-G. Oh.
\newblock Floer cohomology, spectral sequences, and the {M}aslov class of
  {L}agrangian embeddings.
\newblock {\em Internat. Math. Res. Notices}, 1996(7):305--346, 1996.

\bibitem[OU13]{Oa-Ush:certainlags}
J.~Oakley and M.~Usher.
\newblock On certain {L}agrangian submanifolds of ${S}^2\times {S}^2$ and
  $\mathbb{C}{P}^n$.
\newblock Preprint. Can be found at \url{http://arxiv.org/abs/1311.5152}, 2013.

\bibitem[Sch93]{Sc:Morse-homology}
M.~Schwarz.
\newblock {\em Morse homology}, volume 111 of {\em Progress in Mathematics}.
\newblock Birkh\"{a}user Verlag, Basel, 1993.

\bibitem[Sol06a]{Solomon:intersection}
J.~P. Solomon.
\newblock Intersection theory on the moduli space of holomorphic curves with
  {L}agrangian boundary conditions.
\newblock Preprint, can be found at math.SG/0606429, 2006.

\bibitem[Sol06b]{Solomon:thesis}
J.~P. Solomon.
\newblock {\em Intersection theory on the moduli space of holomorphic curves
  with {L}agrangian boundary conditions}.
\newblock ProQuest LLC, Ann Arbor, MI, 2006.
\newblock Thesis (Ph.D.)--Massachusetts Institute of Technology.

\bibitem[ST16]{Sol-Tuk:openGWbounding}
J.~P. Solomon and Sara~B. Tukachinsky.
\newblock Point-like bounding chains in open {G}romov-{W}itten theory.
\newblock Preprint, can be found at \url{https://arxiv.org/abs/1608.02495},
  20016.

\bibitem[Sua14]{Su:exactcob}
L.S. Suarez.
\newblock Exact {L}agrangian cobordism and pseudo-isotopy.
\newblock Preprint, can be found at \url{http://arxiv.org/abs/1412.0697}, 2014.

\bibitem[Sul02]{SullM:kinvariants}
M.~G. Sullivan.
\newblock {$K$}-theoretic invariants for {F}loer homology.
\newblock {\em Geom. Funct. Anal.}, 12(4):810--872, 2002.

\bibitem[Tur01]{Tur:torsionbook}
V.G. Turaev.
\newblock {\em Introduction to combinatorial torsions}.
\newblock Lectures in Mathematics ETH Z\"urich. Birkh\"auser Verlag, Basel,
  2001.
\newblock Notes taken by Felix Schlenk.

\bibitem[Wei94]{Weibel:book-hom-alg}
C.A. Weibel.
\newblock {\em An introduction to homological algebra}, volume~38 of {\em
  Cambridge Studies in Advanced Mathematics}.
\newblock Cambridge University Press, Cambridge, 1994.

\bibitem[Wei13]{Weibel:kbook}
C.A. Weibel.
\newblock {\em The {$K$}-book: {A}n introduction to algebraic $K$-theory},
  volume 145 of {\em Graduate Studies in Mathematics}.
\newblock American Mathematical Society, Providence, RI, 2013.

\bibitem[Wel15]{Welsch:discs4dim}
J.-Y. Welschinger.
\newblock Open {G}romov-{W}itten invariants in dimension four.
\newblock {\em J. Symplectic Geom.}, 13(4):1075--1100, 2015.

\end{thebibliography}
\let\thefootnote\relax\footnote{\noindent This paper is dedicated to the memory of Kevin Henriot.\\
Part of this research was done during my stay at the ETH Z\"urich in 2014.  There, I benefited greatly from numerous discussions with Paul Biran, which helped shape the direction of my research and led me to the results of this paper.  I thank him for his time and interest in this project.  I also thank the Max Planck Institute for Mathematics in Bonn for the great working environment it provides and the financial support.  The idea of using spectral sequences to compute torsion was inspired by a talk of Thomas Kragh at the CAST 2015 workshop in Lyon;  our proof of Theorem \ref{thm:main} is a generalization of Lemma 2.1 from Abouzaid-Kragh \cite{Abou-Kragh:simplehomtypenearby}. I would like to thank the organizers for such a lively conference.  Thanks finally to Saurav Bhaumik for the algebraic discussions related to \S \ref{sec:whiteheadtorsion}.}
\end{document}